%% file: quillenDimension.tex
\documentclass[11pt, a4paper, british]{amsart}

\usepackage{amsmath} 
\usepackage{amsthm} 
\usepackage{amssymb} 
\usepackage[dvipsnames]{xcolor}
\usepackage{mathtools}

\usepackage[cal=boondoxo,scr=euler]{mathalfa}
\usepackage[backref=page,linktocpage]{hyperref} 
\usepackage{newtxtext,newtxmath} 

\DeclareMathAlphabet{\mathsf}{OT1}{\sfdefault}{m}{n}
\newcommand{\nocontentsline}[3]{}
\newcommand{\tocless}[2]{\bgroup\let\addcontentsline=\nocontentsline#1{#2}\egroup}

\usepackage[margin=1.20in]{geometry}
\usepackage{verbatim}
\usepackage{scalerel}
\makeatletter
\def\dual#1{\expandafter\dual@aux#1\@nil}
\def\dual@aux#1/#2\@nil{\begin{tabular}{@{}c@{}}#1\\#2\end{tabular}}
\makeatother

\makeatletter
\@namedef{subjclassname@2020}{\textup{2020} Mathematics Subject Classification}
\makeatother

\DeclareMathAlphabet{\amathbb}{U}{bbold}{m}{n}

\hypersetup{
    colorlinks = true,
    linkbordercolor = {white},
    linkcolor = {BrickRed},
    anchorcolor = {black},
    citecolor = {BrickRed},
    filecolor = {cyan},
    menucolor = {BrickRed},
    runcolor = {cyan},
    urlcolor = {black}
}

\usepackage{enumitem}
\usepackage{pbox}
\usepackage{multirow}
\input{xy}
\xyoption{all}
\xyoption{poly}
\usepackage[all]{xy}

\theoremstyle{plain}
\newtheorem{theorem}{Theorem}[section]
\newtheorem{lemma}[theorem]{Lemma}
\newtheorem{corollary}[theorem]{Corollary}
\newtheorem{proposition}[theorem]{Proposition}

\theoremstyle{definition}
\newtheorem{example}[theorem]{Example}
\newtheorem{definition}[theorem]{Definition}

\theoremstyle{remark}

\numberwithin{equation}{section}

\input{operators}

\title{Maximal subgroups of exceptional groups and Quillen's dimension}

\author[Kevin I. Piterman]{Kevin I. Piterman*}

\address{Philipps-Universit\"at Marburg \\
   Fachbereich Mathematik und Informatik \\
   35032 Marburg, Germany}
\email{piterman@mathematik.uni-marburg.de}

\thanks{*Supported by a postdoctoral fellowship of the Alexander von Humboldt Foundation.}

\begin{document}

\begin{abstract}
Given a finite group $G$ and a prime $p$, let $\Ap(G)$ be the poset of nontrivial elementary abelian $p$-subgroups of $G$.
The group $G$ satisfies the Quillen dimension property at $p$
if $\Ap(G)$ has non-zero homology in the maximal possible degree, which is the $p$-rank of $G$ minus $1$.
For example, D. Quillen showed that solvable groups with trivial $p$-core satisfy this property, and later, M. Aschbacher and S.D. Smith provided a list of all $p$-extensions of simple groups that may fail this property if $p$ is odd.
In particular, a group $G$ with this property satisfies Quillen's conjecture: $G$ has trivial $p$-core and the poset $\Ap(G)$ is not contractible.

In this article, we focus on the prime $p = 2$ and prove that the $2$-extensions of the exceptional finite simple groups of Lie type in odd characteristic satisfy the Quillen dimension property, with only finitely many exceptions.
We achieve these conclusions by studying maximal subgroups and usually reducing the problem to the same question in small linear groups, where we establish this property via counting arguments.
As a corollary, we reduce the list of possible components in a minimal counterexample to Quillen's conjecture at $p = 2$.
\end{abstract}



\maketitle


\section{Introduction}

Since the early 70s, there has been a growing interest in the $p$-subgroup posets and their connections with finite group theory, the classification of the finite simple groups, finite geometries, group cohomology and representation theory.
The poset $\Sp(G)$ of nontrivial $p$-subgroups of a group $G$ was introduced by Kenneth Brown in \cite{Brown}.
In that paper, Brown worked with the Euler characteristic $\upchi(G)$ of groups $G$ satisfying certain finiteness conditions and established connections between the $p$-fractional part of $\upchi(G)$ and the $p$-subgroup structure of $G$.
One of the consequences of his results is the commonly known ``Homological Sylow theorem'', which states that the Euler characteristic of $\Sp(G)$ is $1$ modulo $|G|_p$, the order of a Sylow $p$-subgroup of $G$.

Some years later, Daniel Quillen introduced the poset $\Ap(G)$ of nontrivial elementary abelian $p$-subgroups of a finite group $G$ and exhibited several applications of the topological properties of these posets \cite{Qui78}.
Indeed, the study of elementary abelian $p$-subgroups goes back to Quillen's earlier work on the Bredon cohomology of $G$-spaces and his proof of the Atiyah-Swan conjecture, that relates the Krull dimension of a ring to the dimension of $\Ap(G)$ (see \cite{Qui71}).

In \cite{Qui78}, Quillen showed that $\Sp(G)$ and $\Ap(G)$ are ($G$-equivariantly) homotopy equivalent, and provided a new proof of Brown's result.
In fact, when $G$ is the set of rational points of a semisimple algebraic group over a finite field of characteristic $p$, these posets are homotopy equivalent to the building of $G$ and, hence, they have the homotopy type of a wedge of spheres of dimension $l-1$, where $l$ is the rank of the underlying algebraic group.
Furthermore, in that case, the homology $\widetilde{H_*}(\Ap(G))$ affords the classical Steinberg module for $G$.

Quillen also exhibited other connections between intrinsic algebraic properties of $G$ and the topology of these posets.
For instance, he showed that $\Ap(G)$ is disconnected if and only if $G$ contains a strongly $p$-embedded subgroup.
Recall that the classification of the groups with this property is indeed one of the many important steps towards the classification of the finite simple groups (see for example Section 7.6 of \cite{GLS98}).

On the other hand, Quillen proved that if $G$ has a fixed point on $\Ap(G)$ (or, equivalently on $\Sp(G)$), then these posets are contractible.
Note that $G$ has a fixed point if and only if its $p$-core $O_p(G)$ is nontrivial.
In view of this and further evidence, Quillen conjectured that the reciprocal to this statement should hold.
That is, if $\Ap(G)$ is contractible then there is a fixed point, or, equivalently, $O_p(G)\neq 1$ (see Conjecture 2.9 of \cite{Qui78}).
In other words, Quillen's conjecture asserts that $\Ap(G)$ is contractible if and only if $O_p(G)\neq 1$.

A significant part of Quillen's article is devoted to proving the solvable case of this conjecture.
In \cite{Qui78} it is shown that for a $p$-nilpotent group $G$ with abelian Sylow $p$-subgroups and $O_p(G) = 1$, $\Ap(G)$ is homotopy equivalent to a nontrivial wedge of spheres of the maximal possible dimension, which is $m_p(G) - 1$, the $p$-rank of $G$ minus $1$.
Then, if $G$ is any solvable group with $O_p(G) = 1$, $G$ contains a $p$-nilpotent subgroup $O_{p'}(G)A$, with $A \in\Ap(G)$ of maximal $p$-rank and $O_p(O_{p'}(G)A) = 1$, and thus $\widetilde{H}_{m_p(G)-1}(\Ap(G))\neq 0$.

Later, Michael Aschbacher and Stephen D. Smith formalised this property and gave a name to it: an arbitrary group $G$ with $\widetilde{H}_{m_p(G)-1}(\Ap(G))\neq 0$ is said to satisfy the \textit{Quillen dimension property at $p$}, or $\QD_p$ for short (see \cite{AS93}).
Therefore, a solvable group $G$ with $O_p(G) = 1$ satisfies $\QD_p$ and thus Quillen's conjecture.
Furthermore, it was shown that $p$-solvable groups satisfy this property by using Quillen's techniques and, in addition, the CFSG (see \cite{AD,Smi}).
These results also suggest that a stronger statement of the conjecture may hold: if $O_p(G) = 1$ then $\widetilde{H}_*(\Ap(G);\QQ)\neq 0$.
Therefore, from now on, by Quillen's conjecture we will be referring to this stronger version.

It is not hard to see that not every group $G$ with $O_p(G) = 1$ satisfies $\QD_p$.
For example, we mentioned that finite groups of Lie type in characteristic $p$ satisfy the conjecture, but since the Lie rank is usually strictly smaller than the $p$-rank, they fail $\QD_p$.
This has led to the development of new methods to prove Quillen's conjecture.
One of the most notorious advances in the conjecture was achieved by Aschbacher-Smith in \cite{AS93}.
There, they established Quillen's conjecture for a group $G$ if $p > 5$ and in addition, roughly, all the $p$-extensions of finite unitary groups $\PSU_n(q)$, with $q$ odd and $p \mid q+1$, satisfy $\QD_p$ (see Main Theorem of \cite{AS93} for the precise statement). Here, a $p$-extension of a group $L$ is a split extension of $L$ by an elementary abelian $p$-subgroup of $\Out(L)$.
In \cite{AS93} it is not shown that the group $G$ satisfies $\QD_p$.
Instead, they proved that if every $p$-extension of a fixed component of $G$ satisfies $\QD_p$, then $\widetilde{H}_*(\Ap(G);\QQ)\neq 0$ if $O_p(G) = 1$ (under suitable inductive hypotheses).
This result restricts the possibilities of the components of a minimal counterexample to Quillen's conjecture: every component has a $p$-extension failing $\QD_p$.
In view of this result and the classification of the finite simple groups, Aschbacher and Smith described for $p\geq 3$, all the possible $p$-extensions of simple groups which may potentially fail $\QD_p$.
This is the $\QD$-List, Theorem 3.1, of \cite{AS93}.
Moreover, it is conjectured in \cite{AS93} that the unitary groups $\PSU_n(q)$ with $q$ odd and $p\mid q+1$ should not appear in this list, and so the extra hypothesis on the unitary groups in the main result of \cite{AS93} could be omitted.
Nevertheless, this problem remains open (see \cite{PW} for recent results in this direction).

In the last few years, there have been further developments in the Quillen conjecture \cite{P, PSV, PS1, PS2}.
Recently, in \cite{PS2}, new tools for the study of the conjecture have been provided.
For example, it is shown that the Aschbacher-Smith general approach to the conjecture can be extended to \textit{every} prime $p$ by reducing reliance on results of \cite{AS93} stated only for odd primes and invoking the Classification.
In particular, Theorem 1.1 of \cite{PS2} shows that the Main Theorem of \cite{AS93} extends to $p\geq 3$, keeping the additional constraint on the unitary groups.
On the other hand, for $p = 2$, one important obstruction for this extension is the lack of a $\QD$-List for this prime.
Roughly, Corollary 1.8 of \cite{PS2} concludes that a minimal counterexample to Quillen's conjecture contains a component of Lie type in characteristic $r \neq 3$, and every component of $G$ has a $2$-extension failing $\QD_2$.

In view of these results on Quillen's conjecture, in this article, we focus on showing that the $2$-extensions of the exceptional finite simple groups of Lie type in odd characteristic satisfy $\QD_2$, with a small number of exceptions.
This improves the conclusions of \cite{PS2} on Quillen's conjecture for $p = 2$, and allows us to conclude then that exceptional groups of Lie type in odd characteristic different from $3$ cannot be components of a minimal counterexample to the conjecture (see Corollary \ref{coro:QCprime2} below).

The main result of this article is the following theorem, whose proof is given in different propositions in Section \ref{sec:QDexcp}.

\begin{theorem}
\label{thm:main}
Let $L$ be an exceptional finite simple group of Lie type in odd characteristic.
That is, $L = {}^3D_4(q)$, $F_4(q)$, $G_2(q)$, ${}^2G_2(q)'$, $E_6(q)$, ${}^2E_6(q)$, $E_7(q)$ or $E_8(q)$, with $q$ odd.
Then every $2$-extension of $L$ satisfies the Quillen dimension property at $p = 2$, except possibly in the following cases:
\begin{center}
${}^3D_4(9) $ extended with field automorphisms;
$F_4(3)$; $F_4(9) $ extended with field automorphisms;\\
$2$-extensions of $G_2(3)$; $G_2(9)$ extended with field automorphisms; ${}^2G_2(3)'$;\\
$E_8(3)$; $E_8(9)$ extended with field automorphisms.
\end{center}
\end{theorem}

Indeed, the extensions of $G_2(3)$, $G_2(9)$ and ${}^2G_2(3)'$ mentioned above do fail $\QD_2$ by Example \ref{ex:exceptionsG2} and Proposition \ref{prop:caseRee}.

To achieve the conclusions of Theorem \ref{thm:main}, in most cases we exhibit a maximal subgroup $M$ of a $2$-extension $LB$ of $L$ such that $m_2(M) = m_2(LB)$ and $M$ satisfies $\QD_2$.
Since there is an inclusion $\widetilde{H}_{m_2(LB)-1}(\A_2(M)) \hookrightarrow \widetilde{H}_{m_2(LB)-1}(\A_2(LB))$ in the top-degree homology groups, this establishes $\QD_2$ for $LB$ (see Lemma \ref{lm:inclusion}).
In some cases, the subgroup $M$ arises from suitable parabolic subgroups.
More concretely, when it is possible, we pick $P$ to be a maximal parabolic subgroup of $L$ which is stabilised by $B$ and such that $M := PB$ realises the $2$-rank of $LB$.
Then we get a $2$-nilpotent configuration $UA$, where $U$ is the unipotent radical of $P$, $A$ is an elementary abelian $2$-subgroup realising the $2$-rank of $PB$, and $O_2(UA) = C_A(U) = 1$ by one of the corollaries of the Borel-Tits theorem.
Hence, by Quillen's results on the solvable case, $UA$ satisfies $\QD_2$, and thus also $M$ and $LB$.

When the choice of such parabolic $P$ is not possible, we pick one of the maximal rank subgroups of $L$.
Here, the components of the maximal subgroup $M$ are usually smaller exceptional groups, low-dimensional linear group $A_1(q)$ and $A_2(q)$ or unitary groups ${}^A_2(q)$.
Therefore, we first prove that the $2$-extensions of simple linear and unitary groups in dimensions $2$ and $3$ satisfy $\QD_2$.

Although there is a large literature on maximal subgroups of exceptional groups of Lie type, we will only need the results from \cite{CLSS, KG2q, LSS, LS, LS2}.

Finally, from Theorem \ref{thm:main} and the results of \cite{PS2} for $p = 2$, we can conclude:

\begin{corollary}
\label{coro:QCprime2}
Let $G$ be a minimal counterexample to Quillen's conjecture for $p = 2$.
Then $G$ contains a component of Lie type in characteristic $r \neq 3$.
Moreover, every such component fails $\QD_2$ in some $2$-extension and belongs to one of the following families:
\[ \PSL_n(2^a) (n\geq 3), D_n(2^a) (n\geq 4), E_6(2^a),\]
\[ \PSL^{\pm}_{n} (q) (n\geq 4), B_n(q) (n\geq 2), C_n(q) (n\geq 3), D_n^{\pm}(q) (n\geq 4),\]
where $q = r^a$ and $r > 3$.
\end{corollary}

The $2$-extensions of $\PSL_2(q)$, $\PSL_3(q)$ and $\PSU_3(q)$ satisfy $\QD_2$ by Propositions \ref{casePSL2}, \ref{casePSU3} and \ref{casePSLOddDim} respectively, with exceptions when $q = 3,5,9$.
Nevertheless, the results of \cite{PS2} eliminate these possibilities from a minimal counterexample.

Further results on the Quillen dimension property at $p = 2$ for the classical groups could be pursued by combining the methods presented in this article with the results of \cite{AD,DRM}.

The paper is organised as follows.
In Section 2, we set the notation and conventions that we will need to work with the finite groups of Lie type.
We also provide some useful properties to work out the $p$-extension and compute $p$-ranks.
In Section 3 we gather previous results on the Quillen dimension property and related tools that will help us establish this property.
Then in Section 4 we establish $\QD_2$ for some $2$-extensions of linear groups and recall the structure of the centralisers of graph automorphisms, following Table 4.5.1 of \cite{GLS98}.
In Section 5 we prove each case of Theorem \ref{thm:main}.

All groups considered in this article are finite.
We suppress the notation for the homology coefficients, and we assume that they are always taken over $\QQ$.
The interested reader may note that our results can be extended to homology with coefficients in other rings.
Finally, we emphasise that we adopt the language and conventions of \cite{GLS98}.
This is particularly important when we name the different types of automorphisms of groups of Lie type.
Computer calculations were performed with GAP \cite{GAP4}.

\bigskip

\textbf{Acknowledgements.}
The author thanks Stephen D. Smith for many helpful discussions concerning the algebraic properties of groups of Lie type.
He also thanks Volkmar Welker for his suggestions on a preliminary version of the article.

\section{Preliminaries}
\label{sec:preliminaries}

We assume that the reader is familiar with the construction of the finite groups of Lie type as fixed points of Steinberg endomorphisms, and the basic properties concerning root systems of reductive algebraic groups.
We will follow the language of \cite{GLS98}, which also contains the required background on finite groups of Lie type.
In this section, we will only recall some notations and names, and state results that will be used later.

We denote by $\Cyclic_n$, $\Dihedral_n$, $\Sym_n$ and $\Alt_n$ the cyclic group of order $n$, the dihedral group of order $n$, the symmetric group on $n$ points and the alternating group on $n$ points.

If $G$ is a group, then $\Aut(G)$, $\Inn(G)$ and $\Out(G)$ denote the automorphism group, the group of inner automorphisms and the outer automorphism group of $G$ respectively.
We denote by $Z(G)$ the centre of $G$.
We usually write $G:H$, or simply $GH$, for a split extension of $G$ by $H$.
When an extension of $G$ by $H$ may not split, we denote it by $G.H$.
By an element $g$ (resp. a subgroup $B$) of $G$ inducing outer automorphisms on $L\leq G$ we mean that $g$ embeds into $\Aut(L) \setminus \Inn(L)$ (resp. $B$ embeds in $\Aut(L)$ with $B\cap \Inn(L) = 1$).
Finally, $H \circ_m K$ denotes a central product of $H$ and $K$ by a central cyclic subgroup of order $m$.
That is, $H \circ_m K = (H\times K)/\Cyclic_m$ where $\Cyclic_m$ embeds into both $Z(H)$ and $Z(K)$.

We will usually use the notation $n$ in a group extension to denote a cyclic group of order $n$, and $n^m$ a direct product of $m$ copies of cyclic groups of order $n$.
A number between brackets $[n]$ in the structure description of an extension means some group of order $n$.

In this article, we are mainly interested in extensions by elementary abelian groups.
Below we recall the definition of $p$-extension given in the introduction and introduce some useful notation.

\begin{definition}
\label{def:pextensionAndPouters}
Let $L$ be a finite group and $p$ a prime number.
A $p$-extension of $L$ is a split extension $LB$ of $L$ by an elementary abelian $p$-group $B$ inducing outer automorphisms on $L$.

If $L\leq G$, we denote by $\Outposet_G(L)$ the poset of elements $B\in\Ap(G)$ such that $B\cap LC_G(L) = 1$.
Thus $\Outposet_G(L)$ is the set of $B\in\Ap(G)$ inducing outer automorphisms on $L$.
We write $\Outposet_2(L)$ for $\Outposet_{\Aut(L)}(L)$ at $p = 2$.
We also let $\hat{\Outposet}_G(L) = \Outposet_G(L)\cup \{1\}$ and $\HO(L) = \Outposet_2(L)\cup \{1\}$.
\end{definition}

\begin{definition}
\label{def:quillenDimension}
For a prime number $p$, we say that a group $G$ satisfies the \textit{Quillen dimension property at $p$} if $\Ap(G)$ has non-zero homology in dimension $m_p(G)-1$, where $m_p(G)$ denotes the $p$-rank of $G$:

\vspace{0.2cm}
$\QD_p$ \quad $\widetilde{H}_{m_p(G)-1}(\Ap(G))\neq 0$.
\end{definition}

A remarkable study of the Quillen dimension property for odd primes $p$ was carried out in Theorem 3.1 of \cite{AS93}.
This theorem contains a list of the potential $p$-extensions of simple groups that might fail $\QD_p$, for $p\geq 3$.
In particular, this list contains the $p$-extensions of unitary groups $\PSU_n(q)$ with $q$ odd and $p\mid q+1$.
However, Conjecture 4.1 of \cite{AS93} basically claims that these groups should not belong to this list.
In fact, it is shown there that if $n < q(q-1)$ then these $p$-extensions satisfy $\QD_p$.
Nevertheless, this problem remains open.

\medskip

The aim of this article is to achieve some progress on a similar list for the prime $p = 2$.
Therefore, we will focus on showing that $2$-extensions of certain simple groups satisfy $\QD_2$.
To that end, we introduce the following convenient definition.

\begin{definition}
A group $L$ satisfies (E-$\QD$) if every $2$-extension of $L$ satisfies $\QD_2$:

\vspace{0.2cm}

(E-$\QD$) \quad For every $B\in\HO(L)$, $LB$ satisfies $\QD_2$.
\end{definition}

In order to establish $\QD_p$ for $p$-extensions, it is crucial to be able to compute $p$-ranks of extensions.
The following result, extracted from Lemma 4.2 in \cite{PSV}, will be a useful tool to compute $p$-ranks of extensions.

\begin{lemma}
[$p$-rank of extensions]
\label{lm:prankExtensions}
Let $G = N.K$ be an extension of finite groups, and let $p$ be a prime number.
Then
\[m_p(G) =  \max_{A\in\mathcal{S}} \bigl( m_p(C_N(A)) + m_p(A) \bigr) ,\]
where $\mathcal{S} = \{A\in \Ap(G)\cup \{1\} \tq A\cap N = 1\}$.
In particular, $m_p(G)\leq m_p(N)+m_p(K)$.
\end{lemma}

We will implicitly use this result at many points of the proofs.
Note that, in order to apply this lemma, we should be able to compute centralisers of elementary abelian $2$-subgroups, usually inducing outer automorphisms.
We will often proceed as follows: if $LB$ is a $2$-extension of $L$, then take a suitable decomposition $B = B_0 \oplus B_1$, with $|B_1| = 2$.
Suppose that we can inductively compute the $2$-rank of $LB_0$.
Then, by Lemma \ref{lm:prankExtensions}, we have
\begin{equation}
    \label{eq:standardComputation2rank}
     m_2(LB) = \max \left\{ m_2(LB_0), 1 + m_2(C_{LB_0}(t)) \tq t \in LB\setminus LB_0 \text{ is an involution} \right\}.
\end{equation}
Moreover, this computation depends only on the conjugacy classes of the involutions $t$, and, in most of the cases that we are interested in, such classes are completely classified.

\medskip

Now we recall, rather informally, the names of the different \textit{types} of automorphisms of a simple group of Lie type $K$ defined over a field of odd characteristic, following Definition 2.5.13 of \cite{GLS98}.
We refer to \cite{GLS98} for the full details.
Let $t\in \Aut(K)$ be an involution and $K^* = \Inndiag(K)$.
Then we have the following names for $t$:
\begin{enumerate}
    \item inner-diagonal if $t\in K^*$;
    \item inner if $t\in \Inn(K)$;
    \item diagonal if $t\in K^* \setminus \Inn(K)$;
    \item field automorphism if $t \in \Aut(K)\setminus K^*$ is $\Aut(K)$-conjugated to a field automorphism of the ground field and $K$ is not ${}^2A_n(q)$, ${}^2D_n(q)$ or ${}^2E_6(q)$;
    \item graph if $t \in \Aut(K)\setminus K^*$, roughly, is $\Aut(K)$-conjugated to an involution arising as an automorphism of the underlying Dynkin diagram (except for $K = G_2(q)$), or else from a field automorphism in cases ${}^2A_n(q)$, ${}^2D_n(q)$ and ${}^2E_6(q)$; and
    \item graph-field automorphism if it can be expressed as a product $gf$ of a graph involution $g$ and a field automorphism $f$, or else  $K = G_2(q)$ and $t$ arises from a $\Aut(K)$-conjugated of an involution automorphism of the underlying Dynkin diagram.
\end{enumerate}

It follows from Proposition 4.9.1 of \cite{GLS98} that the centralisers of field involutions $t$ verify that $m_2(C_K(t)) = m_2(K)$ and $m_2(C_{K^*}(t)) = m_2(K^*)$.
By Eq. (\ref{eq:standardComputation2rank}), we see that $m_2(K\gen{t}) = m_2(K)+1$.
Below we reproduce a simplified version of this proposition.

\begin{proposition}
\label{prop:centFieldAut}
Let $K = {}^d\Sigma(q)$ be a group of Lie type in adjoint version in characteristic $r$, and let $x$ be a field or graph-field automorphism of prime order $p$.
Set $K_x = O^{r'}(C_K(x))$.
Then the following hold:
\begin{enumerate}
    \item If $x$ is a field automorphism then $K_x \groupiso {}^d\Sigma(q^{1/p})$.
    \item If $x$ is a graph-field automorphism then $d=1$, $p=2$ or $3$, and $K_x \groupiso {}^p\Sigma(q^{1/p})$.
    \item $K_x$ is adjoint and $C_{\Inndiag(K)}(x) \groupiso \Inndiag(K_x)$.
    \item Fields (resp. graph-field) automorphism are all $\Inndiag(K)$-conjugated, except for graph-fields for $K = D_4(q)$ and $p=3$.
\end{enumerate}
\end{proposition}

The previous proposition does not determine, a priori, the structure of $C_K(x)$, but just of the centraliser taken over the inner-diagonal automorphism group.
Since we are interested in computing $m_2(C_K(x))$, it will be crucial for us to decide when a diagonal involution can centralise a field or graph-field automorphism $x$.
We recall below Lemma 12.8 of \cite[Ch. 17]{GLS8}, which provides a partial solution to this problem.

\begin{lemma}
\label{lm:fieldAndDiag}
Let $K \groupiso \PSL_2(q), \POm_{2n+1}(q), \PSp_{2n}(q)$ or $E_7(q)$, where $q$ is a power of an odd prime $r$.
Let $\phi$ be a field automorphism of order $2$, and let $K_{\phi} = O^{r'}(C_K(\phi))$.
Then $\Inndiag(K_{\phi}) = C_{\Inndiag(K)}(\phi) = C_{\Inn(K)}(\phi)$.
In particular, $\phi$ does not commute with diagonal involutions of $\Inndiag(K)$.
\end{lemma}

We will mainly work with Table 4.5.1 of \cite{GLS98} to compute the $2$-ranks of extensions by diagonal and graph involutions, mostly for the groups of type $A^{\pm}_m(q)$ and the exceptional groups.
In the next paragraph, we briefly and informally describe how to read such table.
See \cite[pp. 171-182]{GLS98} for a complete and accurate description of Table 4.5.1.

This table records the $K^*$-conjugacy classes of inner-diagonal and graph involutions $t$ of a finite group of Lie type $K$ in adjoint version, and the structure of their centralisers when taken over $K^* = \Inndiag(K)$.
The centraliser of an involution $t$ is denoted by $C^* = C_{K^*}(t)$.
The first column of Table 4.5.1 denotes the family for which the involutions are listed ($A_n$, $B_n$, $C_n$, etc.)
The second column indicates the restrictions for these classes to exist, while the third column is a label for the conjugacy class of that involution.
For the purposes of this article, we will not need to interpret the fourth column.
In the fifth column, it is indicated when such classes are of inner type (denoted by $1$), diagonal type (denoted by $d$) or graph type (several notations like $g,g'$).
The notation $1/d$ indicates that it is inner if the condition inside the parentheses at the right holds, and it is diagonal otherwise.
From the sixth column to the end, the structure of the centraliser $C^*$ is described.
Roughly, $C^*$ is an extension of a central product of groups of Lie type $L^* = O^{r'}(C^*)$ (column six), whose versions are specified in the column ``version'' and whose centres can be recovered from the column $Z(L^*)$.
An extra part centralising this product can be computed from the column $C_{C^{*\circ}}(L^*)$.
Here $C^{*\circ} = L^*T^*$ is the connected-centraliser, where $T^*$ is a certain $r'$-subgroup arising from a torus $T$ normalised by $t$ and inducing inner-diagonal automorphisms on $L^*$.
From the columns $L^*$, version, $Z(L^*)$ and $C_{C^{*\circ}}(L^*)$, one can compute the ``inner-part'' of $C^{*\circ}$.
Finally, from the last two columns we can recover the outer automorphisms of $L^*$ arising in $C^{*\circ}$ (in general of diagonal type) and the remaining part of $C^*/C^{*\circ}$, which is often an involution acting on the components of $L^*$ (as field or graph automorphism, or by switching two components) and on the central part $C_{C^*}(L^*)$ (which is usually cyclic and the involution acts by inversion).
To recover the action of the last column, the symbols $i$, $\leftrightarrow$, $\phi$, $\gamma$, $1$ mean, respectively, an action by inversion, a swap of two components, a field automorphism of order $2$, a graph automorphism of order $2$, and an inner action on a component or trivial action on $C_{C^{*\circ}}(L^*)$.

\section{Tools to achieve \texorpdfstring{$\QD_p$}{QDp}}

In this section, we provide tools and collect results that will help us to establish $\QD_2$ on certain $2$-extensions.
Many of these tools were introduced and exploded by Aschbacher-Smith to determine the $\QD$-list in \cite{AS93}.

The following proposition is an easy consequence of the K{\"u}nneth formula for the join of spaces and the fact that $\Ap(H\times K) \simeq \Ap(H) * \Ap(K)$ (see \cite[Prop. 2.6]{Qui78}).

\begin{proposition}
\label{prop:productJoinQDp}
If $p$ is a prime and $H,K$ satisfy $\QD_p$, then $H\times K$ satisfies $\QD_p$.
\end{proposition}

The following lemma corresponds to Lemmas 0.11 and 0.12 of \cite{AS93}.

\begin{lemma}
\label{lm:pprimequotient}
Let $N\leq G$ be such that $N\leq O_{p'}(G)$.
Then there is an inclusion
\[\widetilde{H}_*(\Ap(G/N)) \subseteq \widetilde{H}_*(\Ap(G)).\]
In particular, $m_2(G) = m_2(G/N)$, and if $G/N$ satisfies $\QD_p$ then so does $G$.

If $ N \leq Z(G)$, then indeed $\Ap(G) \equiv \Ap(G/N)$.
\end{lemma}

The following observation is an easy consequence of the inclusion between the homology groups of top-degree.

\begin{lemma}
\label{lm:inclusion}
Let $H\leq G$ be such that $m_p(H) = m_p(G)$.
If $H$ satisfies $\QD_p$, then so does $G$.
\end{lemma}

Next, we recall one of the essential results on the Quillen dimension property.

\begin{theorem}
[Quillen]
\label{thm:solvable}
If $G$ is a solvable group with $O_p(G) = 1$, then $G$ satisfies $\QD_p$.
\end{theorem}

This theorem settles the solvable case of Quillen's conjecture (see \cite[Thm. 12.1]{Qui78}).
Later, it was extended to the family of $p$-solvable groups by using the CFSG if $p$ is odd.
We refer to Chapter 8 of \cite{Smi} for further details on Quillen's conjecture and the Quillen dimension property.

\bigskip

In view of Theorem \ref{thm:solvable} and the Inclusion Lemma \ref{lm:inclusion}, it is convenient to look for solvable subgroups of $G$ with maximal $p$-rank.
Some standard solvable subgroups in a group of Lie type $L$ arise by taking extensions of unipotent radicals by elementary abelian subgroups of their normalisers.
These extensions lie then inside parabolic subgroups.
The following result on parabolic subgroups will help us to achieve (E-$\QD$) for arbitrary groups of Lie type (cf. Step v at p.506 of \cite{AS93}).

\begin{lemma}
\label{lm:QDfromParabolicWithMxlRank}
Let $L$ be a simple group of Lie type, and $p$ a prime not dividing the characteristic of $L$.
Suppose that $LB$ is a $p$-extension of $L$ and that there exists a $B$-invariant parabolic subgroup $P\leq L$ such that $m_p(LB) = m_p(PB)$.
Then $LB$ satisfies $\QD_p$.
\end{lemma}

\begin{proof}
Let $R:=O_r(P)$, where $r$ is the characteristic of the ground field.
Then, as a consequence of the Borel-Tits theorem, $C_{\Aut(L)}(R)\leq R$  (see Corollary 3.1.4 of \cite{GLS98}).
In particular, if $T\leq PB$ realises the $p$-rank of $PB$, then $T$ normalises $R$, and $C_T(R) \leq R\cap T = 1$.
This means that $T$ is faithful on $R$, i.e. $O_p(RT) = 1$, and $m_p(RT) = m_p(PB) = m_p(LB)$.
Then $RT$ is a solvable group with trivial $p$-core, and by Theorem \ref{thm:solvable}, $RT$ satisfies $\QD_p$.
By Lemma \ref{lm:inclusion}, $LB$ satisfies $\QD_p$.
\end{proof}

\begin{lemma}
\label{lm:QDfromParabolicWith2Sylow}
Let $L$ be a simple group of Lie type defined in odd characteristic.
Suppose that $P$ is a proper parabolic subgroup of $L$ containing a Sylow $2$-subgroup of $L$ (that is, $|L:P|$ is odd).
Then $L$ and the extension of $L$ by a field automorphism of order $2$ satisfy $\QD_2$.
\end{lemma}

\begin{proof}
Let $L$ and $P$ be as in the hypotheses of the lemma.
Since $P$ has odd index in $L$, it contains a Sylow $2$-subgroup of $L$.
Therefore, $m_2(P) = m_2(L)$ and by Lemma \ref{lm:QDfromParabolicWithMxlRank}, $L$ satisfies $\QD_2$.

Next, let $B\in \HO(L)$ be cyclic inducing field automorphisms.
By passing through algebraic groups and root systems, it can be shown that $B$ normalises some conjugate of $P$, which we may assume is $P$ itself.
Thus, after conjugation, we suppose that $B\leq N_{\Aut(L)}(P)$.
Note that a Sylow $2$-subgroup of $PB$ is a Sylow $2$-subgroup of $LB$, so $m_2(PB) = m_2(LB)$.
By Lemma \ref{lm:QDfromParabolicWithMxlRank}, $LB$ satisfies $\QD_2$.
\end{proof}

We close this section with a few more results on low $p$-ranks.
The following lemma follows from the $p$-rank $2$ case of Quillen's conjecture.
See \cite[Prop. 2.10]{Qui78}.

\begin{lemma}
\label{lm:connectedDim1}
If $\Ap(G)$ is connected, $m_p(G) = 2$ and $O_p(G) = 1$, then $G$ satisfies $\QD_p$.
\end{lemma}

It will be convenient to recall the classification of groups with a strongly $2$-embedded subgroup, that is, those groups with disconnected $2$-subgroup poset.
See \cite[Thm. 7.6.1]{GLS98} and \cite[Prop. 5.2]{Qui78}.

\begin{theorem}
\label{thm:connectedA_2list}
Let $p = 2$ and $G$ be a finite group.
Then $\A_2(G)$ is disconnected if and only if $O_2(G)=1$ and one of the following holds:
\begin{enumerate}
\item $m_2(G) = 1$;
\item $\Omega_1(G)/O_{p'}(\Omega_1(G))) \groupiso \PSL_2(2^n), \PSU_3(2^n)$ or $\Sz(2^{2n-1})$ for some $n\geq 2$.
\end{enumerate}
In particular, from the isomorphisms among the simple groups, we see that
\[ \Alt_5 \groupiso  \PSL_2(5) \groupiso  \PSL_2(2^2), {}^2G_2(3)' \groupiso \PSL_2(2^3),\]
are included in the list of item (2).
\end{theorem}

Indeed, sometimes in low dimensions, we will be able to conclude $\QD_p$ by computing the sign of the Euler characteristic of $\Ap(G)$.
Therefore, we will use the following well-known expression of this invariant.

\begin{proposition}
\label{propEulerCharQuillenPoset}
The reduced Euler characteristic of $\Ap(G)$ is:
\begin{align*}
\widetilde{\upchi}(\Ap(G)) & = \sum_{\overline{E}\in \Ap(G)/G\cup \{1\}} (-1)^{m_p(E)-1} p^{\binom{m_p(E)}{2}} |G:N_G(E)|.
\end{align*}
\end{proposition}

Finally, the next lemma will help us to produce non-zero homology by inductively looking into the homology of the Quillen poset of a certain normal subgroup and centralisers of outer elements acting on it.
The main reference for this lemma is \cite{SW}.

\begin{lemma}
\label{lm:MVLES}
Let $G$ be a finite group and $p$ a prime number.
Suppose that $L\normal G$ is a normal subgroup such that $\Outposet_G(L)$ consists only of cyclic subgroups.
Then we have a long sequence
\[ \ldots\to \widetilde{H}_{m+1}(\Ap(G)) \to \bigoplus_{B\in \Outposet_G(L)} \widetilde{H}_m(\Ap(C_L(B))) \overset{i_*}{\to} \widetilde{H}_m(\Ap(L)) \overset{j_*}{\to} \widetilde{H}_m(\Ap(G))\to\ldots \]
where $i_*$ and $j_*$ are the natural maps induced by the inclusions $\Ap(C_L(B)) \subseteq \Ap(L)$ and $\Ap(L)\subseteq \Ap(G)$, respectively.

In particular, the following hold:
\begin{enumerate}
\item Let $X$ be the union of the subposets $\Ap(C_N(B))$ for $B\in \Outposet_G(L)$.
We have indeed a factorisation
\begin{equation}
\label{eq:factorisationMV}
\xymatrix{
 \bigoplus_{B\in \Outposet_G(L)} \widetilde{H}_m(\Ap(C_L(B))) \ar[rr]^{i_*} \ar[rd]^{i'_*} & & \widetilde{H}_m(\Ap(L))\\
  & \widetilde{H}_m(X) \ar[ru]^{k_*} &
}
\end{equation}
where also $i'_*$ and $k_*$ are induced by the inclusions $\Ap(C_L(B))\subseteq X$ and $X \subseteq \Ap(L)$, respectively.
\item $m_p(G) \leq m_p(L)+1$.
\item If $\widetilde{H}_{m_p(L)-1}(\Ap(C_L(B))) = 0$ for all $B\in \Outposet_G(L)$, then $H_{m_p(L)}(\Ap(G)) = 0$.
\item We have a bound
\begin{align*}
\dim H_{m_p(L)}(\Ap(G)) & \geq \sum_{B\in \Outposet_G(L)}\dim \widetilde{H}_{m_p(L)-1}(\Ap(C_L(B))) - \dim \widetilde{H}_{m_p(L)-1}(X)\\
& \geq  \sum_{B\in \Outposet_G(L)}\dim \widetilde{H}_{m_p(L)-1}(\Ap(C_L(B))) - \dim \widetilde{H}_{m_p(L)-1}(\Ap(L)).
\end{align*}
\item If $m_p(G) = m_p(L)+1$ and $G$ fails $\QD_p$, then, for $m = m_p(L)-1$, we get inclusions
\[ \bigoplus_{B\in \Outposet_G(L)} \widetilde{H}_m(\Ap(C_L(B))) \hookrightarrow \widetilde{H}_m(X) \hookrightarrow \widetilde{H}_m(\Ap(L)) .\]
\end{enumerate}
\end{lemma}

\begin{proof}
The long exact sequence arises from the main result of \cite{SW}.
Then Eq. (\ref{eq:factorisationMV}) in item (1) is an immediate consequence of this sequence.
Item (2) holds by Lemma \ref{lm:prankExtensions}.
Items (3-5) follow by looking into the last terms of the long exact sequence, at $m = m_p(L)$.
\end{proof}

\section{Some linear groups satisfy \texorpdfstring{$\QD_2$}{QD2}}

In this section, we prove that the linear groups $\PSL_2(q)$ and $\PSL_{3}(q)$ satisfy (E-$\QD$) for every $q$, with a few exceptions for $q=3,5,9$.
These cases will serve as basic cases for the exceptional groups, where we will occasionally find linear groups as direct factors in some of their maximal subgroups.

From \cite[Prop. 4.10.5]{GLS98}, we recall the $2$-ranks of the small dimensional linear groups:

\begin{proposition}
\label{prop:2ranksSmallLinearGroups}
If $q$ is a power of an odd prime and $n=2,3$, then $\PSL_n^\pm(q)$ and $\PGL_n^\pm(q)$ have $2$-rank $2$.
\end{proposition}

We begin by studying the linear group of dimension $2$.

\begin{proposition}
\label{casePSL2}
Let $L\groupiso \PSL_2(q)$ with $q$ odd and $q\neq 3$.
Then every $2$-extension $LB$ of $L$ satisfies $\QD_2$, with the following exceptions:
\begin{enumerate}
\item $L \groupiso \PSL_2(5)$, $B=1$;
\item $L \groupiso \PSL_2(9)$, $B$ induces field automorphisms of order $2$.
\end{enumerate}
Moreover, every $2$-extension of $\Inndiag(L)\groupiso \PGL_2(q)$ satisfies $\QD_2$, except in case (2).
\end{proposition}

\begin{proof}
We consider the possible $2$-extensions of $L$.
In any case, we know that $L$ is simple and that $\Out(L) = \Cyclic_2\times \Cyclic_a$, where $\Cyclic_2\groupiso \Outdiag(L)$ and $\Cyclic_a$ is the group of field automorphisms of $\GF{q}$.
Suppose that $\phi$ is an order $2$-field automorphisms of $\GF{q}$ (if it exists), and that $d\in \Inndiag(L) \setminus L$ is a diagonal involution.
Then the $2$-extensions of $L$ are given in Table \ref{tab:2extPSL}.
\begin{table}[ht]
\begin{tabular}{|c|c|c|}
\hline
$2$-extension $LB$ & $C_L(B)$ & $m_2(LB)$ \\
\hline
$B = 1$ & $L$ & $2$\\
$B = \gen{\phi}$ & $\PGL_2(q^{1/2})$ & $3$\\
$B = \gen{d}$ & $\Dihedral_{q+\epsilon}$ & $2$\\
\hline
\end{tabular}
\caption{$2$-extensions of $\PSL_2(q)$, $q\geq 5$ odd. Here $q\equiv \epsilon \pmod{4}$, $\epsilon\in \{1,-1\}$.}
\label{tab:2extPSL}
\end{table}
This table follows since every involution of $\Aut(\PSL_2(q)) - \PGL_2(q)$ is a field automorphism.
Recall also that field and diagonal automorphisms of order $2$ do not commute by Lemma \ref{lm:fieldAndDiag}
The structure of the centraliser for $d$ follows from the first row of Table 4.5.1 of \cite{GLS98}.
Finally, observe that $L\gen{d} = \Inndiag(L)$ and $m_2(\Inndiag(L)\gen{\phi}) = 3$ since $m_2(L) = m_2(\Inndiag(L)) = 2$.

We prove that each $2$-extension of $L$ satisfies $\QD_2$ by computing the Euler characteristic.
First, $2$-extensions $LB$ and $\Inndiag(L)\gen{\phi}$ have connected $\A_2$-poset by Theorem \ref{thm:connectedA_2list}, except for $L = \PSL_2(5)$, $B=1$.
Therefore, by Lemma \ref{lm:connectedDim1}, $L$ and $\Inndiag(L)$ satisfy $\QD_2$, except for $L =\PSL_2(5)$.
Note that $\A_2(\PSL_2(5)) = \A_2(\Alt_5) = \A_2(\PSL_2(4))$ is homotopically discrete with $5$ points, and the $2$-extension $\PGL_2(5)\groupiso \Sym_5$ does satisfy $\QD_2$.
This yields the conclusions of the statement for the case $q = 5$.

Next we show $\QD_2$ for the $2$-extensions $L\gen{\phi}$ and $\Inndiag(L)\gen{\phi}$, both of $2$-rank $3$ by Lemma \ref{lm:inclusion}.
Thus, it is enough to show that $L\gen{\phi}$ satisfies $\QD_2$.
In order to do this, we compute the dimensions of $H_1(\A_2(L))$ and $H_1(\A_2(C_L(\phi)))$.

Since in this situation $q$ is a square, $q\neq 5$.
Second, if $q = 25$, $C_L(\phi) = \PGL_2(5)$.
Hence, in any case, the dimension of these degree $1$ homology groups can be computed from the reduced Euler characteristic of the underlying $\A_2$-poset.
Here we use the formula given in Proposition \ref{propEulerCharQuillenPoset}.
Thus, for $K = L$ or $C_L(\phi)$,
\begin{align}
\begin{split}
\label{eq:formulaEulerPSL}
\dim H_1(\A_2(K)) & = -\widetilde{\upchi}(\A_2(K))\\
& = 1 - \text{\# of involutions in }K + 2 \cdot \text{\# of $4$-subgroups of } K.    
\end{split}
\end{align}
In Table \ref{tab:invAnd4sbgrps} we describe these numbers:
\begin{table}[ht]
\begin{tabular}{|c|c|c|}
 \hline
 Group & Number of involutions & Number of $4$-subgroups \\
 \hline
 $\PSL_2(q)$ & $\frac{q(q+\epsilon)}{2}$ & $\frac{q(q^2-1)}{24}$ \\
  $\PGL_2(q)$ & $q^2$ & $\frac{q(q^2-1)}{6}$\\
  \hline
\end{tabular}
\caption{Here $q\equiv \epsilon\pmod{4}$, $\epsilon\in \{1,-1\}$.}
\label{tab:invAnd4sbgrps}
\end{table}

\begin{proof}
[Proof of Table \ref{tab:invAnd4sbgrps}]
The number of involutions and $4$-subgroups of $\PSL_2(q)$ follows from Dickson's classification of the subgroups of $\PSL_2(q)$ (see also Theorem 6.5.1 of \cite{GLS98}).

The number of involutions of $\PGL_2(q)$ follows since there is a unique conjugacy class of diagonal involutions $d$ by Table 4.5.1 of \cite{GLS98}.
Thus, the number of elements in such conjugacy class is equal to $\frac{q(q-\epsilon)}{2}$, which gives $q^2$ after adding the number of involutions in $\PSL_2(q)$.

Finally, to compute the number of four-subgroups of $\PGL_2(q)$ we proceed as follows: each four-subgroup of $\PGL_2(q)$ is either contained in $\PSL_2(q)$ or else it contains a unique involution of $\PSL_2(q)$ and $2$ diagonal involutions.
Therefore, for a given diagonal involution $d$, there is a one-to-one correspondence between $4$-subgroups containing $d$ and involutions in $C_L(d) \groupiso \Dihedral_{q+\epsilon}$.
This shows that each diagonal involution is contained in $(q+\epsilon)/2$ $4$-subgroups.
Since we have $\frac{q(q-\epsilon)}{2}$ diagonal involutions, the total number of $4$-subgroups in $\PGL_2(q)$ containing diagonal involutions is
\[ \frac{q(q-\epsilon)}{2}\cdot \frac{(q+\epsilon)}{2} \cdot \frac{1}{2} = \frac{q(q^2-1)}{8}.\]
Thus the total number of $4$-subgroups in $\PGL_2(q)$ is
\[ \frac{q(q^2-1)}{24} + \frac{q(q^2-1)}{8} = \frac{q(q^2-1)}{6}.\]
This completes the proof of Table \ref{tab:invAnd4sbgrps}.
\end{proof}

Indeed, by Table \ref{tab:invAnd4sbgrps}, we get concrete values for the dimensions of the degree $1$ homology groups of $\A_2(\PSL_2(q))$ and $\A_2(\PGL_2(q))$:
\begin{equation}
    \label{eq:dimH1PSL2}
    \dim H_1(\A_2(\PSL_2(q))) = - \widetilde{\upchi}(\A_2(\PSL_2(q))) =
     \frac{1}{12} (q - \epsilon) (q^2 - (6-\epsilon) q - \epsilon 12),
\end{equation}
\begin{equation}
    \label{eq:dimH1PGL2}
    \dim H_1(\A_2(\PGL_2(q))) = - \widetilde{\upchi}(\A_2(\PGL_2(q))) = \frac{1}{3} (q - 3) (q^2 - 1).
\end{equation}
Now we need to describe the number of field automorphisms in $\PSL_2(q)\gen{\phi}$ and in $\PGL_2(q)\gen{\phi}$.

Recall that the field automorphisms of $\PSL_2(q)\gen{\phi}$ are all $\PGL_2(q)$-conjugated, with centraliser $C_{\PGL_2(q)}(\phi) = C_{\PSL_2(q)}(\phi)$.
Thus, the number of field automorphisms of order $2$ in $\PSL_2(q)\gen{\phi}$ is exactly
\[ \frac{|\PGL_2(q)|}{|C_{\PSL_2(q)}(q)|} = \frac{q(q^2-1)}{q^{1/2}(q-1)} = q^{1/2}(q+1).\]
This gives $q^{1/2}(q+1)$ involutions in $\PSL_2(q)\gen{\phi} \setminus \PSL_2(q)$.
Let $L=\PSL_2(q)$, $B = \gen{\phi}$.
By Lemma \ref{lm:MVLES}, the values in Table \ref{tab:invAnd4sbgrps} and formula (\ref{eq:formulaEulerPSL}), we conclude that:
\begin{align*}
\dim H_2(\A_2(LB)) & \geq q^{1/2}(q+1)\dim H_1(\A_2(\PGL_2(q^{1/2}))) - \dim H_1(\A_2(\PSL_2(q)))\\
& = q^{1/2}(q+1)\frac{1}{3}(q^{1/2}-3)(q-1) - \frac{1}{12}(q-1)(q^2-5q-12)\\
& = \frac{1}{4} (q^{1/2} - 1) (q - 1) (q^{3/2} - 3 q - 4).
\end{align*}
Note that $q\equiv 1 \pmod{4}$.
The above number is positive for all $q\geq 13$, which is our case since $q$ is an even power of an odd prime and $q\neq 9$ by hypothesis.
We conclude that $LB = \PSL_2(q)\gen{\phi}$ satisfies $\QD_2$.
Then also $\PGL_2(q)\gen{\phi}$ satisfies $\QD_2$.
Moreover,
\begin{align}
\begin{split}
\label{eq:dimH2PGL2andPSL2withFields}
\dim H_2(\A_2(\PGL_2(q)\gen{\phi})) & \geq \dim H_2(\A_2(\PSL_2(q)\gen{\phi}))\\
& \geq \frac{1}{4} (q^{1/2} - 1) (q - 1) (q^{3/2} - 3 q - 4).
\end{split}
\end{align}

We have shown that every possible $2$-extension of $\PSL_2(q)$ and $\PGL_2(q)$ satisfies $\QD_2$, except for the cases described in the statement of the theorem.
\end{proof}

We note that the excepted cases in Proposition \ref{casePSL2} actually fail $\QD_2$.
Indeed, $\PSL_2(5)$ fails $\QD_2$ since it has $2$-rank $2$ and $\A_2(\PSL_2(5)) = \A_2(\PSL_2(4))$ is homotopically discrete.
The following example provides the details that show that $\PSL_2(9)\gen{\phi}$ and $\PGL_2(9)\gen{\phi}$ fail $\QD_2$, where $\phi$ is a field automorphism of order $2$.

\begin{example}
\label{casePSL29}
Let $L = \PSL_2(9)$ and let $A = \Aut(L)$.
Then $A/L\groupiso \Cyclic_2\times \Cyclic_2$, so every $2$-extension of $L$ is a nontrivial normal subgroup of $A$.
This gives $3$ possible $2$-extensions of $L$, but not $4$.
Let $\phi$ be a field automorphism of $L$ and $d$ a diagonal automorphism of $L$, both of order $2$.
Then the possible $2$-extensions of $L$ are:
\begin{enumerate}
\item $L$, with $2$-rank $2$, satisfies $\QD_2$ with $H_1(\A_2(L))$ of rank $16$;
\item $L \gen{\phi}$, with $2$-rank $3$, fails $\QD_2$ since $C_L(\phi)\groupiso \Sym_4$, which has nontrivial $2$-core $O_2(C_L(\phi)) \groupiso \Cyclic_2\times \Cyclic_2\neq 1$;
\item $L\gen{d}=\PGL_2(9)$, with $2$-rank $2$, satisfies $\QD_2$ with $H_1(\A_2(L)\gen{d})$ of rank $160$ and $C_L(d) \groupiso \Dihedral_{10}$.
\end{enumerate}
Note that $\Aut(L)$ has $2$-rank $3$ and does not satisfy $\QD_2$, and it is not a $2$-extension of $L$ since diagonal and field automorphisms do not commute in $\Aut(L)$.
Also $\PGL_2(9)\gen{\phi}$ fails $\QD_2$ since $C_{\PGL_2(9)}(\phi) = C_L(\phi)$ has nontrivial $2$-core.

There is also a remaining almost simple group $N$ with $L<N<\Aut(L)$, not contained in the previous cases.
This is the extension $N = \PSL_2(9).2 \groupiso \Alt_6.2$, and it satisfies that $\A_2(N) = \A_2(L)$.
Therefore, although this group $N$ is not a $2$-extension of $L$, it is a ``non-split $2$-extension'', and it does satisfy $\QD_2$.

Finally, these computations show that $\A_2(L) \hookrightarrow \A_2(\Aut(L))$ induces an inclusion in homology, and hence a non-zero map.
By the main result of \cite{PS1}, $\PSL_2(9)$ is not a component of a minimal counterexample to Quillen's conjecture.
\end{example}

Our next aim is to show that $2$-extensions of $\PSL_3(q)$ satisfy $\QD_2$, with only a few exceptions.
We will need the following lemma which records the values of the Euler characteristic of the Quillen poset of some linear groups and the unitary groups in dimension $3$.

\begin{lemma}
\label{eulerChar}
For $L = \PSL_n(q)$ and $n$ odd, we have
\[ \widetilde{\upchi}(\A_2(L)) = \widetilde{\upchi}(\A_2(\PGL_n(q))) = \frac{(-1)^n}{n} \prod_{i=1}^{n-1} (q^i-1) f_n(q),\]
where $f_n(q)$ denotes a polynomial as described in \cite{Welker}. For instance, $f_3(q) = q^3+3q^2+3q+3$.
Moreover, since $\A_2(L)$ is Cohen-Macaulay of dimension $n-2$, the above Euler characteristic computes the dimension of $H_{n-2}( \A_2(L) )$.

If $L = \PSU_3(q)$, then
\[ \widetilde{\upchi}(\A_2(L)) = \widetilde{\upchi}(\A_2(\PGU_3(q))) = \frac{1}{3}(q^6-2q^5-q^4+2q^3-3q^2+3). \]
\end{lemma}

\begin{proof}
The value of the Euler characteristic for $\PGL_n(q)$ follows from Proposition 4.1 and Theorem 4.4 of \cite{Welker} (note that there is a typo in the formula of Theorem 4.4, and the product over $i$ should be up to $r-1$).
Also, since $n$ is odd, by Proposition 7.5 of \cite{PW}, $\A_2(\PSL_n(q)) = \A_2(\PGL_n(q)) = \A_2(\GL_n(q))_{>Z}$ where $Z$ is the cyclic subgroup of order $2$ of $Z(\GL_n(q))$.
By \cite{Qui78} (see also \cite{Welker}), $\A_2(\PSL_n(q))$ is Cohen-Macaulay of dimension $n-2$.

The formula for $\PGU_3(q)$ follows from Example 7.6 of \cite{PW}.
\end{proof}

Next, we show that the $2$-extensions of $\PSU_3(q)$ satisfy $\QD_2$, except for $q = 3$.
These cases will be important during our analysis for $\PSL_3(q)$, especially when working with $2$-extensions by graph-field automorphisms.

\begin{proposition}
\label{casePSU3}
Let $L = \PSU_3(q)$. Then $L$ satisfies (E-$\QD$) if $q\neq 3$.
Moreover, let $\phi$ be a graph automorphism of order $2$ of $L$.
Then we have
\begin{align*}
\dim H_2(\A_2(&\PGU_3(q)\gen{\phi})) \geq \dim H_2(\A_2(\PSU_3(q) \gen{\phi}))\\
& \geq \frac{1}{3}(q^2-1)(q+1)\left( \frac{q^2(q^2-q+1)}{(3,q+1)}(q-3) - (q^3-3q^2+3q-3) \right),
\end{align*}
which is a positive polynomial for $q > 3$.
Finally, for $q = 3$, $\PSU_3(3)$ satisfies $\QD_2$ but $\PSU_3(3)\gen{\phi}$ fails $\QD_2$.
\end{proposition}

\begin{proof}
We have that $\A_2(L)$ is connected by Theorem \ref{thm:connectedA_2list}, and $m_2(L) = 2$ by Proposition \ref{prop:2ranksSmallLinearGroups}.
Thus $L$ satisfies $\QD_2$ by Lemma \ref{lm:connectedDim1}.
Moreover, by Lemma \ref{eulerChar},
\begin{equation}
\label{eq:dimH1PSU3}
\dim H_1(\A_2(L)) = -\widetilde{\upchi}(\A_2(L)) = \frac{1}{3}(q^6-2q^5-q^4+2q^3-3q^2+3).
\end{equation}

Next, the only possible nontrivial $2$-extension of $L$ is by a graph automorphism $\phi$ of order $2$ (which indeed arises from the field automorphism $x\mapsto x^q$).
Let $L_1 = L\gen{\phi}$ be such extension.
By Table 4.5.1 of \cite{GLS98},
\[ C_{\PGU_3(q)}(\phi) \groupiso \Inndiag(\Omega_3(q)) = \PGL_2(q).\]
This implies that $C_L(\phi) = \PGL_2(q)$.
Moreover, there is a unique $\PGU_3(q)$-conjugacy class of graph automorphisms, and such elements act by inversion on $\Outdiag(L) = (3,q+1)$.
Thus the conjugacy class of $\phi$ in $\Out(L)$ has size $(3,q+1)$, and this gives rise to exactly $(3,q+1)$ extensions $L\gen{\phi'}\leq \Aut(L)$ of $L$ by a conjugate $\phi'$ of $\phi$, and these extensions are $\Aut(L)$-conjugated.
We conclude then that the number of graph automorphisms contained in $L_1$ is
\[ n_g := \frac{|\PGU_3(q)|}{|\PGL_2(q)| (3,q+1)} = \frac{q^2(q^3+1)}{(3,q+1)}.\]

Finally, by Lemma \ref{lm:MVLES}, we conclude that
\begin{align*}
\dim H_2(\A_2(&\PGU_3(q)\gen{\phi})) \geq \dim H_2(\A_2(\PSU_3(q) \gen{\phi}))\\
&\quad \geq n_g \dim H_1(\A_2(\PGL_2(q))) - \dim H_1(\A_2(\PSU_3(q)))\\
&\quad = \frac{q^2(q^3+1)}{(3,q+1)}\frac{1}{3}(q-3)(q^2-1)   - \frac{1}{3}(q^6-2q^5-q^4+2q^3-3q^2+3)\\
&\quad = \frac{1}{3}(q^2-1)(q+1)\left( \frac{q^2(q^2-q+1)}{(3,q+1)}(q-3) - (q^3-3q^2+3q-3) \right).
\end{align*}
This polynomial is positive for all $q > 3$.
Therefore, $L_1$ satisfies $\QD_2$ if $q\neq 3$.

When $q = 3$, $C_L(\phi) = \PGL_2(3)$ has nontrivial $2$-core, so $H_1(\A_2(C_L(\phi))) = 0$, and by Lemma \ref{lm:MVLES}, $H_2(\A_2(L_1)) = 0$.
\end{proof}

Now we have the necessary background to prove that $\PSL_3(q)$ satisfies (E-$\QD$), except for a small number of cases.

\begin{proposition}
\label{casePSLOddDim}
Let $L = \PSL_n(q)$ with $n,q$ odd.
The following assertions hold:
\begin{enumerate}
\item $L$ and $L$ extended by a field involution satisfy $\QD_2$.
\item If $n = 3$, then every $2$-extension of $L$ satisfies $\QD_2$, with the following exceptions that fail $\QD_2$:
\begin{itemize}
\item $L = \PSL_3(3)$ extended by a graph automorphism, and
\item $L = \PSL_3(9)$ extended by a group generated by a field involution and a graph automorphism.
\end{itemize}
\end{enumerate}

\end{proposition}

\begin{proof}
Let $L = \PSL_n(q)$, with $n$ odd, and consider the stabiliser $P$ of a $1$-dimensional subspace of the underlying module $V = \GF{q}^n$.
Then $P$ is a parabolic subgroup with structure $P \groupiso [q^{n-1}]L_P$, where $L_P$, a Levi complement for $P$, has structure $\SL_{n-1}(q) \circ_{(n,q-1)} \Cyclic_{q-1}$.
Thus $|L_P| = |\GL_{n-1}(q)| / (n,q-1)$ and the index of $P$ in $L$ is:
\[ |L:P| = \frac{q^{n(n-1)/2}\prod_{i=2}^{n}(q^i-1)}{q^{n-1}\cdot  q^{(n-1)(n-2)/2}\prod_{i=1}^{n-1}(q^i-1)} = \frac{q^n-1}{q-1} = q^{n-1} + q^{n-2} + \ldots + q + 1.\]
Since $n$ is odd, the index of $P$ in $\PSL_n(q)$ is odd.
By Lemma \ref{lm:QDfromParabolicWith2Sylow}, $L=\PSL_n(q)$ and $L$ extended by a field involution satisfy $\QD_2$.
This proves item (1).

Before moving to the case $n = 3$, we list all the possible $2$-extensions of $L$.
Denote by $\phi$, $\gamma$ and $\delta$ a field automorphism of order $2$, a graph automorphism and a graph-field automorphism of $L$, respectively, such that $[\phi,\gamma] = 1$ and $\delta = \phi \gamma$.
Let also $L^* = \PGL_n(q)$.
Then the $2$-extensions of $L$ are:
\begin{enumerate}[label=(\roman*)]
\item $L$;
\item $L\gen{\phi}$, with $C_{L^*}(\phi) \groupiso \PGL_n(q^{1/2})$ by Proposition \ref{prop:centFieldAut};
\item $L\gen{\gamma}$, with $C_{L}(\gamma) \groupiso \Inndiag(\Omega_n(q))$ by Table 4.5.1 of \cite{GLS98};
\item $L\gen{\delta}$, with $C_{L^*}(\delta) \groupiso \PGU_n(q^{1/2})$ by Proposition \ref{prop:centFieldAut};
\item $L\gen{\phi,\gamma}$, with $C_{L}(\phi,\gamma) \groupiso \Inndiag(\Omega_n(q^{1/2}))$ by (iii) and Proposition \ref{prop:centFieldAut}.
\end{enumerate}
Now suppose that $n = 3$, that is $L = \PSL_3(q)$.
We know that the extensions of cases (i) and (ii) above satisfy $\QD_2$ by the parabolic argument.
So it remains to show that the $2$-extensions by graph, graph-field and both graph and field automorphisms, satisfy $\QD_2$.
To that end, we compute the dimensions of the top-degree homology groups, similar to what we did for $\PSL_2(q)$ in the proof of Proposition \ref{casePSL2}.

First, recall that we have the following number of involutions of each type.
Let $B = \gen{\phi,\gamma}$.
\begin{align*}
n_f & := \# \text{ field involutions in }L\gen{\phi} = \#\text{ field involutions in }LB\\
& = \frac{|\PGL_3(q)|}{|\PGL_3(q^{1/2})| (3,q^{1/2}+1)},\\
n_g & := \#\text{ graph involutions in }L\gen{\gamma} = \#\text{ graph involutions in }LB\\
& = \frac{|\PGL_3(q)|}{|\PGL_2(q)| (3,q-1)},\\
n_{gf} & := \#\text{ graph-field involutions in }L\gen{\delta} = \#\text{ graph-field involutions in }LB \\
& = \frac{|\PGL_3(q)|}{|\PGU_3(q^{1/2})| (3,q^{1/2}-1)}.
\end{align*}
To compute these numbers, we have used the structure of the centraliser in each case, the fact that there is a unique $L^*$-conjugacy class for each type of involution, and the structure of $\Out(L) = (3,q-1) : \gen{\phi,\gamma}$ (cf. Theorem 2.5.12 of \cite{GLS98}).

Let $t$ be a field, graph or graph-field involution of $L$, and let $L_1 = L\gen{t}$.
Then the number $n_1$ of involutions in $L_1 \setminus L$ is $n_f,n_g$ or $n_{gf}$, accordingly to the type of $t$.
Note also that $m_2(L_1) = m_2(L) + 1 = 3$.

By Lemma \ref{lm:MVLES},
\begin{equation}
\label{eq:boundDimH2}
\dim H_2(\A_2(L_1)) \geq n_1 \cdot \dim H_1(\A_2(C_L(t))) - \dim H_1(\A_2(L)).
\end{equation}
We compute $d(t) := \dim H_1(\A_2(C_L(t)))$ in each case, by using Lemma \ref{eulerChar} and Eq. (\ref{eq:dimH1PGL2}).
Note that $\Omega_1(C_L(\phi)) = \PSL_3(q^{1/2})$ by item (ii) above.
Also $C_L(\gamma) = \PGL_2(q)$ by the classical isomorphism $\Inndiag(\Omega_3(q)) \groupiso \PGL_2(q)$.
By Lemma \ref{eulerChar}, we have:
\begin{align*}
  d(\phi) & = \dim H_1(\A_2(\PSL_3(q^{1/2}))) = \frac{1}{3}(q^{1/2}-1)(q-1)(q^{3/2}+3q+3q^{1/2}+3),\\
d(\gamma) & = \dim H_1(\A_2(\PGL_2(q)))       = \frac{1}{3}(q-3)(q^2-1),\\
d(\delta) & = \dim H_1(\A_2(\PGU_3(q^{1/2}))) = \frac{1}{3}(q^3-2q^{5/2}-q^2+2q^{3/2}-3q+3).\\
\end{align*}
Let $d := \dim H_1(\A_2(L))$.
By Eq. (\ref{eq:dimH1PSL2}), this dimension is
\[ d = \frac{1}{12}(q-\epsilon)(q^2-(6-\epsilon)q-\epsilon 12),\]
with $q\equiv \epsilon\pmod{4}$ and $\epsilon \in \{\pm 1\}$.

Now it is routine to verify that $n_1 d(t) > d$ if $t = \gamma$ or $t = \delta$, if and only if $(t,q) \neq (\gamma,3)$.
Indeed, for $q = 3$, $C_L(\gamma) = \PGL_2(3) \groupiso \Sym_4$ has non-trivial $2$-core, so $d(\gamma) = 0$ and in consequence, $H_2(L\gen{\gamma}) = 0$.
This shows that $L\gen{\gamma}$ fails $\QD_2$ if $q = 3$.
Therefore, a $2$-extension of $L$ by a field, graph or graph-field involution satisfies $\QD_2$ if and only if $q \neq 3$ when $L$ is extended by a graph involution.

It remains to show that $LB = L\gen{\phi,\gamma}$ verifies $\QD_2$.
For this case, we take $L_f = L\gen{\phi}$, $L_2 = LB$ and consider the long exact sequence of Lemma \ref{lm:MVLES} at $m = 2$ there (since $m_2(L_2) = 4$).
That is, we need to show that $H_3(\A_2(L_2)) \neq 0$.

Note that the set of involutions $t\in L_2\setminus L_1$ is exactly the set of all graph and graph-field automorphisms of the extension $L_2 = LB$.
Let $d_g := \dim  H_2(\A_2(\PGL_2(q)\gen{\phi}))$, $d_{gf} :=  \dim H_2(\A_2(\PGU_3(q^{1/2})\gen{\phi}))$ and $d_f := \dim H_2(\A_2(L_f))$.
Therefore, by Lemma \ref{lm:MVLES},
\begin{equation}
\dim H_3( \A_2 (L_2)) \geq n_g d_g  + n_{gf}d_{gf} - d_f.
\end{equation}
We show that the right-hand side of this equation is positive if $q\neq 9$ by providing proper bounds of the dimensions $d_g$, $d_{gf}$ and $d_f$.

By Eq. (\ref{eq:dimH2PGL2andPSL2withFields}),
\begin{equation}
d_g = \dim H_2(\A_2(\PGL_2(q)\gen{\phi})) \geq \frac{1}{4}(q^{1/2}-1)(q-1)(q^{3/2}-3q-4).
\end{equation}
Next, by Proposition \ref{casePSU3},
\begin{equation}
d_{gf} \geq \frac{1}{3}(q-1)(q^{1/2}+1)\left( \frac{q(q-q^{1/2}+1)}{(3,q^{1/2}+1)}(q^{1/2}-3) - (q^{3/2}-3q+3q^{1/2}-3) \right),
\end{equation}
which is positive for all $q > 9$.

Finally, we need to bound  $d_f$ from above.
Indeed, by Lemma \ref{lm:MVLES} at $m = 2$, we have
\begin{align*}
d_f & = \dim H_2(\A_2(L_f)) = \dim H_2(\A_2(\PSL_3(q)\gen{\phi}))\\
& \leq n_f \dim H_1(\A_2(\PSL_3(q^{1/2})))\\
& = \frac{q^{3/2}(q+1)(q^{3/2}+1)}{(3,q^{1/2}+1)}.
\end{align*}
Now we check with the given bounds that $n_g d_g + n_{gf} d_{gf} - d_f$ is positive if and only if $q > 9$.
Indeed, if $q = 9$, similar arguments show $H_3(\A_2(LB)) = 0$ since $d_g = 0$ by Example \ref{casePSL29} and $d_{gf} = 0$ by Proposition \ref{casePSU3}.

We conclude that every $2$-extension of $\PSL_3(q)$ satisfies $\QD_2$, except for $\PSL_3(3)$ extended by a graph automorphism and for $\PSL_3(9)$ extended by field and graph automorphisms, which actually fail $\QD_2$.
\end{proof}

\section{The Quillen dimension property on exceptional groups of Lie type}
\label{sec:QDexcp}

We use the results of the preceding sections to show that, with only finite exceptions, the $2$-extensions of the exceptional groups of Lie type satisfy $\QD_2$.
For that purpose, it will be convenient to recall first which $2$-extension can arise in each case.
Table \ref{tab:outerExceptionals} records the $2$-ranks of the exceptional groups of Lie type in adjoint version and the structure of the outer automorphism group.
From this, we can compute the possible $2$-extensions in each case.
Recall that we follow the terminology of \cite{GLS98}.
The $2$-ranks were extracted from \cite{CSeitz} and \cite[Prop. 4.10.5]{GLS98}.

\begin{table}[ht]
    \centering
    \begin{tabular}{|c|c|c|c|}
        \hline
        Group & $2$-rank & $\Outdiag$ & $\Out / \Outdiag$\\
        \hline
        ${}^3D_4(q)$ & $3$& $1$ & $3\Phi$ \\
        $G_2(q)$ & $3$ & $1$ & \pbox{8.5cm}{\centering $\Phi\Gamma$, where $|\Phi\Gamma:\Gamma|=2$ if $q=3^a$, and $\Gamma = 1$ otherwise}
        \\
        ${}^2G_2(q)$ & $3$ & $1$ & $\Phi$ (odd order) \\
        $F_4(q)$ & $5$ & $1$ & $\Phi$ \\
        $E_6(q)$ & $6$ & $(3,q-1)$ & $\Phi\times \Gamma$, $\Gamma \groupiso \Cyclic_2$ \\
        ${}^2E_6(q)$ & $6$ & $(3,q+1)$ & $2\Phi$  \\
        $E_7(q)$ & $8$ & $2$ & $\Phi$ \\
        $E_8(q)$ & $9$ & $1$ & $\Phi$ \\
        \hline
    \end{tabular}
    \caption{$\Out/\Outdiag$ is cyclic unless specified; $\Phi = \Aut(\GF{q})\groupiso \Cyclic_a$, where $q = r^a$, $r$ is an odd prime, and the usual conventions for the twisted types hold. Also, $\Gamma$ is a set of graph automorphisms.}
    \label{tab:outerExceptionals}
\end{table}

\subsection{Cases \texorpdfstring{$G_2(q)$}{G2(q)} and \texorpdfstring{${}^2G_2(q)$}{2G2(q)}.}

We start by proving that the Ree groups ${}^2G_2(q)$ satisfy $\QD_2$ if and only if $q \neq 3$.
Note that, by Table \ref{tab:outerExceptionals} for example, ${}^2G_2(q)$ has no non-trivial $2$-extension.

\begin{proposition}
\label{prop:caseRee}
Let $L$ be the Ree group ${}^2G_2(q)$, where $q$ is a power of $3$ by an odd positive integer.
Then the following hold:
\begin{enumerate}
    \item $L$ has no non-trivial $2$-extensions.
    \item A Sylow $2$-subgroup of $L$ is an elementary abelian group of order $8$, so $m_2(L) = 3$.
    \item $2$-subgroups of equal order of $L$ are conjugated.
    \item $L$ satisfies $\QD_2$ if and only if $q \neq 3$.
    Moreover, if $q >3$ then
    \begin{equation}
    \label{eq:dimH2case2G2}
    \dim H_2(\A_2(L)) \geq  \widetilde{\upchi}(\A_2(L)) = \frac{1}{21} (q^2 - 1) (q^5 - 8 q^4 + 15 q^3 + 21) > 0.
\end{equation}
    \item For $q = 3$, $\A_2(L) = \A_2(\PSL_2(8))$ is homotopy equivalent to a discrete space of $8$ points.
\end{enumerate}
\end{proposition}

\begin{proof}
Items (1-3) are well-known facts about the Ree groups and can be found in \cite{Ward}.

If $L = {}^2G_2(3)$, then $L' = \PSL_2(8)$ has index $3$ in $L$, and $\A_2(L) \cong \A_2(\PSL_2(8))$ is homotopy equivalent to a discrete space with $8$ points.
Since $m_2(L) = 3$, we conclude that $L$ fails $\QD_2$ for $q=3$.
This proves item (5) and the ``only if'' part of item (4).

Now suppose that $q\neq 3$ and $L = {}^2G_2(q)$.
Since $\A_2(L)$ has dimension $2$ by item (2), we show that its second homology group is non-zero.
To that end, it is enough to see that its Euler characteristic is positive since $\A_2(L)$ is connected for $q\neq 3$ by Theorem \ref{thm:connectedA_2list}.
Indeed, $\widetilde{\upchi}(\A_2(L)) = \dim H_2(\A_2(L)) - \dim H_1(\A_2(L)) \leq \dim H_2(\A_2(L))$.

We invoke Theorem C of \cite{KG2q} to describe the normalisers of $2$-subgroups:
the centraliser of an involution is $2\times \PSL_2(q)$, the normaliser of a four-subgroup is $(2^2\times \Dihedral_{\frac{q+1}{2}}):3$, and the normaliser of a Sylow $2$-subgroup is $2^3:7:3$.
From this information, items (2,3) and Proposition \ref{propEulerCharQuillenPoset}, we can compute the Euler characteristic of $\A_2(L)$:
\begin{align*}
    \widetilde{\upchi}(\A_2(L)) & = -1 + \frac{|L|}{2|\PSL_2(q)|} - 2 \frac{|L|}{6(q+1)} + 8 \frac{|L|}{168}\\
    & =  -1 + q^3(q^3+1)(q-1) \left( \frac{1}{q(q^2-1)} - \frac{1}{3(q+1)} + \frac{1}{21} \right) \\
    & = \frac{1}{21} (q^2 - 1) (q^5 - 8 q^4 + 15 q^3 + 21).
\end{align*}
Since the polynomial $q^5 - 8 q^4 + 15 q^3 + 21$ is positive for every prime power $q \neq 4$, we conclude that $H_2(\A_2(L))\neq 0$.
In consequence, $L$ satisfies $\QD_2$ if $q\neq 3$.
This completes the proof of item (4), and hence of this proposition.
\end{proof}

For the case $G_2(q)$, we refer the reader to the classification of maximal subgroups of $G_2(q)$ by P. Kleidman \cite{KG2q}.
We will follow the terminology of that article.

\begin{proposition}
\label{caseG2}
Let $L = G_2(q)$, with $q$ odd.
Then every $2$-extension of $L$ satisfies $\QD_2$, except possibly for the $2$-extensions of $G_2(3)$ and the $2$-extension of $G_2(9)$ by a field involution.
\end{proposition}

\begin{proof}
Let $L = G_2(q)$.
We prove first that $G_2(q)$ and its extension by a field automorphism of order $2$ satisfy $\QD_2$, by exhibiting a maximal subgroup of the same rank that satisfies $\QD_2$.

By Theorem A in \cite{KG2q}, $G_2(q)$ contains a subgroup $K_+ = \SL_3(q):2$.
Let $L_+ = F^*(K_+) \groupiso \SL_3(q)$ and $Z = Z(L_+)$.
Then $L_0:= L_+/Z = \PSL_3(q)$ and $H_0:=K_+/Z = L_0 \gen{\gamma}$, where $\gamma$ induces a graph automorphism on $L_0$ (see Proposition 2.2 and its proof in \cite{KG2q}).
By Proposition \ref{casePSLOddDim}, $L_0$ satisfies $\QD_2$ if $q\neq 3$, so $H_0$ satisfies $\QD_2$.

On the other hand, $m_2(L) = 3$ by Table \ref{tab:outerExceptionals}, and also $m_2(L_0) = 3$ by the proof of Proposition \ref{casePSLOddDim}.
Recall from Lemma \ref{lm:pprimequotient} that
\[ \widetilde{H}_*(\A_2(H_0)) = \widetilde{H}_*(\A_2(K_+/Z)) \subseteq \widetilde{H}_*(\A_2(K_+)).\]
In particular, we get the following inclusions between the top-degree homology groups
\[ \widetilde{H}_2(\A_2(H_0)) \subseteq \widetilde{H}_2(\A_2(K_+))\subseteq \widetilde{H}_2(\A_2(L)),\]
which show that $L$ satisfies $\QD_2$ if $q\neq 3$.

Next, a nontrivial $2$-extension of $L=G_2(q)$ can only be given by field automorphisms of order $2$ if $q$ is not a power of $3$.
Moreover, by the construction of the subgroup $K_+$ given in \cite{KG2q}, field automorphisms of $G_2(q)$ induce field automorphisms on (a suitable conjugate of) $K_+$, and hence on the quotient $H_0$.
Thus, for $B\in \Outposet_2(L)$ inducing field automorphisms, we may take $K_+$ fixed by $B$, and then $K_+B\groupiso \SL_3(q):(2\times B)$ after a suitable choice of conjugates (recall that $\Out(\SL_3(q)) = (3,q-1):(\Aut(\GF{q})\times \Gamma)$, where $\Gamma = 2$ is a group of graph automorphisms).
Similar as before, we have a split extension $K_+B /Z = L_0B'$, where $B'=\gen{\gamma}\times B \in \Outposet(L_0)$.
By Proposition \ref{casePSLOddDim}, $L_0B'$ satisfies $\QD_2$ if $q\neq 9$.
Analogously to the previous case, $m_2(L_0B') = 4 = m_2(L) = m_2(K_+B)$, and we get an inclusion in the degree $3$ homology groups, showing that $K_+B$ and $LB$ satisfy $\QD_2$.
Therefore, an extension of $L$ by a field automorphism of order $2$ satisfies $\QD_2$ if $q\neq 9$.

It remains to analyse the case $q = 3^a$.
By Table 4.5.1 of \cite{GLS98} (see also Theorem 2.5.12 of \cite{GLS98}), only field or graph-field automorphisms can arise in $\Aut(L)$.
We have shown above that the extension of $L$ by a field automorphism of order $2$ satisfies $\QD_2$ if $q\neq 9$.
Thus we need to prove that if $t$ is a graph-field automorphism of $L$, then $L\gen{t}$ satisfies $\QD_2$.
In that case, $q = 3^{2a+1}$ and by Proposition \ref{prop:centFieldAut}, $C_{L}(t) = {}^2G_2(q)$, which has $2$-rank $3$.
Therefore $m_2(L\gen{t}) = 4$.
However, by Theorem B of \cite{KG2q}, every maximal subgroup of $L\gen{t}$ containing $t$ is either $2$-local or has $2$-rank at most $3$.
This shows that we cannot proceed as before via maximal subgroups.
In view of this, we will proceed by using the long exact sequence of Lemma \ref{lm:MVLES}.

We have subgroups $M_0:=C_L(t) = {}^2G_2(q)$, $M_1:= G_2(3)\gen{t} \leq L\gen{t}$ and $M_2 := {}^2G_2(3)$ such that $M_2\leq M_1\cap M_0$.
Fix $A$ a Sylow $2$-subgroup of $M_2$.
By Proposition \ref{prop:caseRee}(2) and \cite[Lm. 2.4]{KG2q}, $A$ is also a Sylow $2$-subgroup of $M_0$ and it is self-centralising in $L$, i.e. $C_L(A) = A$.
A direct computation also shows that $N_{M_2}(A) = A.\PSL_3(2)$, which immediately implies $N_L(A) = A.\PSL_3(2)$.

Now, suppose by the way of contradiction that $L\gen{t}$ fails $\QD_2$, that is, the homology group $H_3(\A_2(L\gen{t}))$ vanishes.
Recall that $C_L(t) = {}^2G_2(q)$ and there is a unique $L$-conjugacy class of involutions $t'\in L\gen{t} - L$ by Proposition \ref{prop:centFieldAut}(4).
Let $X = \bigcup_{C_L(t)x \in L/C_L(t)} \A_2(C_L(t^x))$.
By Lemma \ref{lm:MVLES}, we get inclusions
\begin{equation}
\label{eq:caseG2withgraphs}
    \bigoplus_{L/C_L(t)} H_2(\A_2(C_L(t))) \hookrightarrow  H_2 \left(X \right) \hookrightarrow H_2(\A_2(L)).
\end{equation}
Set
\begin{align*}
   d  & := \dim H_2(X), \\
   d' & := \dim \bigoplus_{L/C_L(t)} H_2(\A_2(C_L(t))) = |L:C_L| \dim H_2(\A_2({}^2G_2(q))) .
\end{align*}
Eq. (\ref{eq:caseG2withgraphs}) shows that $d' \leq d$.
However, we will prove that $d < d'$, arriving then at a contradiction.

On one hand, we have that $X$ is a union of $\A_2$-posets. Therefore, below each point, we have a wedge of spheres of maximal possible dimension.
This means that the homology of $X$ can be obtained from the chain complex that in degree $i$ is freely generated by the spheres below each point of $X$ of height $i$.
In particular, for $i = 2$, the points of height $2$ correspond to the conjugates of $A$, the fixed Sylow $2$-subgroup of $M_0 = C_L(t)$ and $M_2$.
Thus,
\begin{align*}
    d = \dim{} &{} H_2(X) \leq |L:N_L(A)| \cdot  \# \text{ spheres below $A$} = \frac{q^6(q^6-1)(q^2-1)}{168}.
\end{align*}

On the other hand, by Proposition \ref{prop:caseRee}(4),
\[ d' \geq |L:C_L(t)| \cdot \widetilde{\upchi}(\A_2({}^2G_2(q))) =
q^3(q^3-1)(q+1)\frac{1}{21}(q^2-1)(q^5-8 q^4+15q^3+21) .\]
Finally, from these bounds for $d$ and $d'$, it is not hard to conclude that $d'>d$ for all prime power $q\geq 7$, which is our case.
This gives a contradiction to Eq. (\ref{eq:caseG2withgraphs}), and thus shows that $H_3(\A_2(L\gen{t})) \neq 0$, that is, $L\gen{t}$ satisfies $\QD_2$.
This finishes the proof of the proposition.
\end{proof}

\begin{example}
\label{ex:exceptionsG2}
Let $L = G_2(3)$.
We show that $\A_2(L)$ is homotopy equivalent to a wedge of spheres of dimension $1$.
In particular, since $m_2(L) = 3$, $L$ fails $\QD_2$.
Moreover, by Lemma \ref{lm:MVLES}, also the unique nontrivial $2$-extension of $L$ (by a graph-field automorphism) fails $\QD_2$.

We construct a subposet of $\A_2(L)$ of dimension $1$ and homotopy equivalent to $\A_2(L)$.
First, take the subposet $\mathfrak{i}(\A_2(L)) = \{A\in \A_2(L) \tq A= \Omega_1(Z(\Omega_1(C_L(A))))\}$, which is homotopy equivalent to $\A_2(L)$ (see \cite[Rk. 4.5]{Pwebb}).
Next, there are two conjugacy classes of elementary abelian $2$-subgroups of order $8$, and both are contained in $\mathfrak{i}(\A_2(L))$.
For one of these classes, say represented by $A$, the normalizer $N_L(A)$ has order $192$.
Then it can be shown that $\mathfrak{i}(\A_2(L))_{<A}$ is contractible.
Therefore, if we remove the $L$-conjugates of $A$ from $\mathfrak{i}(\A_2(L))$ we get a subposet $\mathfrak{s}\mathfrak{i}(\A_2(L))$ homotopy equivalent to $\mathfrak{i}(\A_2(L))$.
Now, there is a unique conjugacy class of four-subgroups in this new subposet $\mathfrak{s}\mathfrak{i}(\A_2(L))$, and each such subgroup is contained in a unique element of order $8$ of $\mathfrak{s}\mathfrak{i}(\A_2(L))$.
Again, we can remove all the four-subgroups from $\mathfrak{s}\mathfrak{i}(\A_2(L))$ and obtain a new subposet $T$ homotopy equivalent to $\A_2(L)$.
Since $T$ consists only of elements of order $2$ and $8$, we conclude that $T$ has dimension $1$.
Finally, an extra computation shows that $\widetilde{\upchi}(\A_2(L)) = -11584$.
Therefore $\A_2(L)$ is homotopy equivalent to a wedge of $11584$ spheres of dimension $1$.
In particular, $L$ fails $\QD_2$.

This also shows that $L = G_2(9)$ extended by a field automorphism of order $2$ fails $\QD_2$: if $\phi$ is a field involution, then $C_L(\phi) = G_2(3)$, and thus $H_2(\A_2(C_L(\phi))) = 0$ by the previous computation. Then by Lemma \ref{lm:MVLES}, we conclude that $H_3(\A_2(L)) = 0$.
\end{example}

\subsection{Cases \texorpdfstring{${}^3D_4$}{3D4(q)} and \texorpdfstring{$F_4(q)$}{F4(q)}}

\begin{proposition}
\label{case3D4}
The group $L = {}^3D_4(q)$ satisfies (E-$\QD$) if $q\neq 9$ is odd.
Also ${}^3D_4(9)$ satisfies $\QD_2$.
\end{proposition}

\begin{proof}
Recall that $m_2(L) = 3$ by Table \ref{tab:outerExceptionals}.
Then a graph automorphism of order $3$ of ${}^3D_4(q)$ centralises a subgroup $K = G_2(q)$.
Also, if $\phi$ denotes a field automorphism of order $2$ of $L$, then, after choosing a suitable conjugate, we may assume that $\phi$ induces a field automorphism on $K$.
By Proposition \ref{caseG2} and its proof, $m_2(K) = 3 = m_2(L)$, $m_2(K\gen{\phi}) = 4 = m_2(L\gen{\phi})$, and both $K$ and $K\gen{\phi}$ satisfy $\QD_2$ for $q\neq 3,9$ respectively.
Also note that $G_2(9)$ satisfies $\QD_2$.
By Lemma \ref{lm:inclusion}, $L$ and $L\gen{\phi}$ satisfy $\QD_2$ if $q\neq 3,9$ respectively.
Since these are the only possible $2$-extensions of $L$ by Table \ref{tab:outerExceptionals}, this concluded with the proof of our proposition for $q\neq 3$.

If $q = 3$ then a computation of the Euler characteristic of $L$ in GAP shows that $\widetilde{\upchi}(\A_2(L)) = 882634225472$.
Since $\A_2(L)$ is connected by Theorem \ref{thm:connectedA_2list}, we see that $H_2(\A_2(L))\neq 0$, that is, $L$ satisfies $\QD_2$.
\end{proof}

\begin{proposition}
\label{caseF4}
If $L = F_4(q)$, with $q\neq 3,9$ odd, then $L$ satisfies (E-$\QD$).
Also $F_4(9)$ satisfies $\QD_2$.
\end{proposition}

\begin{proof}
Suppose that $q\neq 3,9$ is an odd prime power.
Then $L$ contains a subgroup $H:=\PGL_2(q)\times G_2(q)$ (cf. the main result of \cite{LS2}).
Note that $H$ satisfies $\QD_2$ by Propositions \ref{prop:productJoinQDp}, \ref{casePSL2} and \ref{caseG2}.
Since both $L$ and $H$ have $2$-rank $5$ by Table \ref{tab:outerExceptionals}, we conclude that $L$ satisfies $\QD_2$.

Let $B \in \Outposet_2(L)$, so $B$ is generated by a field automorphism of order $2$.
Thus it acts by field automorphisms in a direct product subgroup isomorphic to $H$, which we may assume without loss of generality that it is our $H$.
Then $\widetilde{H} = \PGL_2(q)B \times G_2(q^{1/2})$, which is a subgroup of $HB$, satisfies $\QD_2$ by Propositions  \ref{prop:productJoinQDp}, \ref{casePSL2} and \ref{caseG2}.
Since $m_2(\widetilde{H}) = 6 = m_2(LB)$, we conclude that $LB$ also satisfies $\QD_2$.

We have shown that every possible $2$-extension of $L$ satisfies $\QD_2$, so $L$ satisfies (E-$\QD$).

If $q = 9$, then $\PGL_2(9)\times G_2(9)$ satisfies $\QD_2$ by Propositions \ref{prop:productJoinQDp}, \ref{casePSL2} and \ref{caseG2}.
Therefore, $F_4(9)$ satisfies $\QD_2$.
\end{proof}

\subsection{Cases \texorpdfstring{$E_6(q)$ and ${}^2E_6(q)$}{E6(q) and 2E6(q)}.}

\begin{proposition}
\label{caseE6}
Let $L = E^{\epsilon}_6(q)$ (any version), $\epsilon\in \{\pm 1\}$, and $q$ odd.
Then $L$ satisfies (E-$\QD$).
\end{proposition}

\begin{proof}
Let $L = E^\epsilon_6(q)$ in adjoint version, where $\epsilon \in\{\pm 1\}$.
For a $2$-extension $LB$ of the adjoint version $L$, we see that $m_2(LB) = m_2(L_u\widetilde{B})$, where $L_u$ is the universal version of $E^{\epsilon}_6(q)$ and $\widetilde{B}$, isomorphic to $B$, is just a lift of the action of $B$ on $L_u$ (this is possible since $Z(L_u) = (3,q-\epsilon)$ is odd).
Thus $LB=L_u\widetilde{B}/Z(L_u)$, and by Lemma \ref{lm:pprimequotient}, $\widetilde{H}_*(\A_2(LB)) \subseteq \widetilde{H}_*(\A_2(L_uB))$.
Therefore, if $L$ satisfies (E-$\QD$), then so does the universal version of $E^{\epsilon}_6(q)$.

We will show that there exists a parabolic subgroup $P$ of $L$ such that for any $2$-extension $LB$, a suitable conjugate of $P$ is normalised by $B$ (so we can suppose it is $P$ itself), and $m_2(PB) = m_2(LB)$.

This parabolic subgroup $P$ arises from the $A_5$ subdiagram in $E_6$, so $P = U \GL_6^\epsilon(q)/Z(L_u)$, where $\GL_6^\epsilon(q)/Z(L_u)$ denotes the Levi complement.
Then $m_2(P) = 6$, which realises the $2$-rank of $L$.
Furthermore, a graph, graph-field or field automorphism of $L$ (the last two only for $\epsilon=1$) stabilises this subdiagram (and hence $P$), inducing a graph (resp. graph-field or field) automorphism on $\GL_6^\epsilon(q)/Z(L_u)$.
Denote by $t$ such automorphism.
Then $m_2(L\gen{t}) \leq m_2(L)+1=7$.
We claim that
\begin{equation}
    \label{eq:E6ranksParabolics}
     m_2(P\gen{t}) = m_2(\GL^\epsilon_6(q)\gen{t}) = 7 = m_2(L\gen{t}). 
\end{equation}
Note that $m_2(P\gen{t}) = m_2(\GL_6^\epsilon(q)\gen{t})$, for the lifted action of $t$ on $\GL_6^\epsilon(q)$.
Then it is clear that Eq. (\ref{eq:E6ranksParabolics}) holds if $t$ induces a field automorphism (so $\epsilon = 1$), since the stabiliser of $t$ in $\GL_6(q)$ is $\GL_6(q^{1/2})$.
Similarly, if $t$ is a graph-field automorphism then $\epsilon = 1$ and $C_{\GL_6(q)}(t) = \GU_6(q^{1/2})$, which has $2$-rank $6$.
Then, in these two situations, $m_2(P\gen{t}) = 7$.

Now assume that $t$ is a graph involution.
For $\epsilon = 1$, $t$ acts on $\GL_6(q)$, so $\GL_6(q)\gen{t}$ contains a graph automorphism $g$ inducing the map $x\mapsto (x')^{-1}$, where $x'$ denotes the transpose of $x$.
Therefore, $C_{\GL_6(q)}(g) = \GO_6(q)$, which has $2$-rank $6$.
This implies that $m_2(\GL_6(q)\gen{t}) = 6$.
If $\epsilon = -1$, $t$ is a graph involution acting on $\GU_6(q)$, so up to conjugation $t$ is indeed the map $x\mapsto x^q$.
Therefore, $C_{\GU_6(q)}(t) = \GO_6(q)$, and again we get $m_2(\GU_6(q)\gen{t}) = 6$.
In any case, we see that $m_2(P\gen{t}) = 7$.

Finally, suppose that we have $B = \gen{\phi,\gamma}$, where $\phi$ is a field automorphism of order $2$ and $\gamma$ a graph automorphism of order $2$ of $L = E_6(q)$.
We can suppose that $B$ stabilises $P$ (and thus its unipotent radical), and its Levi complement $\GL_6(q)$.
Thus, $\gamma$ acts as a graph automorphism on the stabiliser of $\phi$ in $\GL_6(q)$, which is isomorphic to $\GL_6(q^{1/2})$.
As we saw above, $m_2(\GL_6(q^{1/2})\gen{\gamma}) = 7$.
Therefore, $m_2(\GL_6(q)B) = 8$.
Since $m_2(B) = 2$ and $m_2(E_6(q)) = 6$, we conclude that $m_2(E_6(q)B) = 8$, so the $2$-rank is realised in $PB$.

To conclude, note that a $2$-extension of $L$ is one of:
\begin{enumerate}
    \item $L$, of $2$-rank $6 = m_2(P)$,
    \item $L\gen{\gamma}$ of $2$-rank $7$, with $\gamma$ a graph automorphism of order $2$, which also stabilises $P$ and $m_2(P\gen{\gamma}) = 7$,
    \item $L\gen{\phi}$ of $2$-rank $7$, with $\phi$ a field automorphism of order $2$ ($\epsilon=1$), which also stabilises $P$ and $m_2(P\gen{\phi}) = 7$,
    \item $L\gen{\gamma\phi}$ of $2$-rank $7$, with $\gamma\phi$ a graph-field automorphism of order $2$ ($\epsilon=1$), which also stabilises $P$ and $m_2(P\gen{\gamma\phi}) = 7$,
    \item $L\gen{\gamma,\phi}$ of $2$-rank $8$, with $\phi$ a field automorphism of order $2$ ($\epsilon=1$) commuting with $\gamma$ a graph automorphism of order $2$, and $\gen{\gamma,\phi}$ also stabilises $P$ with $m_2(P\gen{\gamma,\phi}) = 8$.
\end{enumerate}

From this, we conclude that any $2$-extension of the adjoint versions of $E_6^\epsilon(q)$ satisfies $\QD_2$.
By the remark at the beginning of the proof, we conclude that any version of $E^\epsilon_6(q)$ satisfies (E-$\QD$).
\end{proof}

\subsection{Case \texorpdfstring{$E_7(q)$}{E7(q)}.}

\begin{proposition}
Let $L = E_7(q)$ (adjoint version), with $q$ odd.
Then $L$ satisfies (E-$\QD$).
\end{proposition}

\begin{proof}
Let $L = E_7(q)$.
By Table \ref{tab:outerExceptionals}, if $\phi$ denotes a field automorphism of order $2$ of $L$, the $2$-extensions of $L$ are:
\[L, \quad \Inndiag(L), \quad L\gen{\phi}.\]
Note that $\Inndiag(L)\gen{\phi}$ is not a $2$-extension since field and diagonal automorphisms of order $2$ do not commute in view of Lemma \ref{lm:fieldAndDiag}.

Next, we study the $2$-ranks of these extensions, so we need to understand the centralisers of the outer involutions.
From Table \ref{tab:outerExceptionals}, $m_2(L) = 8$.
We claim that $m_2(\Inndiag(L)) = 8 = m_2(L)$.
Indeed, consider $K = E_7(q^2)$ in adjoint version.
Then $m_2(K) = 8$.
Let $\phi'$ be a field automorphism of order $2$ for $K$.
Then, by Proposition \ref{prop:centFieldAut} and Lemma \ref{lm:fieldAndDiag},
\[ K \geq C_K(\phi') = C_{\Inndiag(K)}(\phi') = \Inndiag(E_7(q)) \groupiso \Inndiag(L).\]
From this we see that $m_2(\Inndiag(L)) = 8 = m_2(L)$.
In particular, $\Inndiag(L)$ satisfies $\QD_2$ if $L$ does.
Moreover, this also proves that if $\phi$ is a field automorphism of order $2$ for $L$ then
\[ m_2(\Inndiag(L)\gen{\phi}) = 9 = m_2(L\gen{\phi}).\]

From these observations, we conclude that, in order to establish (E-$\QD$) for $E_7(q)$, it is enough to show that $E_7(q)$ and $E_7(q)\gen{\phi}$ satisfy $\QD_2$.

To this end, we exhibit a maximal parabolic subgroup of $E_7(q)$ of $2$-rank $8$.
We see that $D_6$ is a subdiagram of $E_7$, so we have a maximal parabolic subgroup in $E_7(q)$ of the form
\[ P = U : (D_6(q).(q-1)).\]
Here $U$ denotes the unipotent radical of $P$, and the subgroup $H = D_6(q)$ is a quotient of $\Spin_{12}^+(q)$ by a central subgroup of order $2$.
Indeed, $H = \HSpin_{12}^+(q)$ and it lies in the centraliser of the involution that generates the centre of a Sylow $2$-subgroup $T$ of $L$ (see the $t_1$ involution of the $E_7(q)$ entry in Table 4.5.1 of \cite{GLS98}).
From this, we show that the Levi complement $L_P = D_6(q).(q-1)$ of $U$ has $2$-rank $8$.
Let $t$ be the involution in the centre of $L_P$.
Then $C_L(t) = (\SL_2(q) \circ_2 \HSpin_{12}^+(q)).2$ by Table 4.5.1 of \cite{GLS98}.
Since $t\in Z(T)$, $T\leq C_L(t)$.
Also, $\SL_2(q)$ has a unique involution, so the $2$-rank of $T$ is realised in a subgroup of the extension $M := \HSpin_{12}^+(q).2$.
The $2$ here at the end comes from diagonal automorphisms of the half-spin group, as in the Levi complement above.
Therefore, if we identify $M$ as a subgroup of $L_P$, we conclude that $m_2(L_P) = m_2(M) = m_2(E_7(q))$.

Moreover, after suitable choices of conjugates, a field automorphism $\phi$ of order $2$ must normalise $P$ and act as a field automorphism on our $M$.
Since $C_M(\phi)$ contains a subgroup isomorphic to $\HSpin_{12}^+(q^{1/2}).2$, we see that $P\gen{\phi}$ has $2$-rank $9$, which is the $2$-rank of the $2$-extension $E_7(q)\gen{\phi}$.

By Proposition \ref{lm:QDfromParabolicWithMxlRank}, $L$ and $L\gen{\phi}$ satisfy $\QD_2$.
Finally, by the previous discussion, we conclude that $L$ satisfies (E-$\QD$).
\end{proof}

\subsection{Case \texorpdfstring{$E_8(q)$}{E8(q)}.}

\begin{proposition}
\label{caseE8}
The simple group $E_8(q)$, $q\neq 3,9$ odd, satisfies (E-$\QD$).
Also $E_8(9)$ satisfies $\QD_2$.
\end{proposition}

\begin{proof}
Let $L = E_8(q)$.
By Table 5.1 of \cite{LSS}, $L$ contains a maximal subgroup
\[ H\groupiso (3,q-1).(\PSL_3(q)\times E_6(q)).(3,q-1).2.\]
Note that $F^*(H) = (3,q-1).(\PSL_3(q)\times E_6(q))$, and
$H_+ := H/Z(F^*(H)) = (\PSL_3(q)\times E_6(q)).(3,q-1).2$, where $(3,q-1)$ induces diagonal automorphism on each component of $H_+$, and the $2$ induces a graph involution, also acting on both components.
In particular, by taking the centraliser of a graph involution on the $\PSL_3(q)$ component, we see that $H_0$ contains a subgroup $K_0$ isomorphic to
\[ \PGL_2(q) \times \Inndiag(E_6(q))\gen{\gamma},\]
where $\gamma$ is a graph involution of $E_6(q)$ centralising $\PGL_2(q)$.
Now, recall that $m_2(L) = 9$ and $m_2(\PGL_2(q)) = 2$.
Since $ m_2(E_6(q)\gen{\gamma})  = 7$ by item (2) of the proof of Proposition \ref{caseE6}, we see that
\[ m_2(K_0) = m_2(\PGL_2(q)) + m_2(E_6(q)\gen{\gamma}) = 2 + 7 = 9 = m_2(L).\]
Therefore $K_0$ realises the $2$-rank of $L$.

By Table \ref{tab:outerExceptionals}, $E_8(q)$ extended by a field automorphism of order $2$, say $\phi$, is the unique nontrivial $2$-extension.
From the construction of the maximal subgroup $H$ and $K_0$ (cf. \cite{LSS}), we can pick a suitable $L$-conjugate of $\phi$ (and we suppose it is the same $\phi$) such that it normalises $H$ and, after passing to the quotient, normalises $K_0$ and induces a field automorphism on both factors of $K_0$.
In particular, we have a subgroup $K_1$ of $K_0 \gen{\phi}$ of the form
\[ \PGL_2(q^{1/2}) \times \Inndiag(E_6(q)) \gen{\gamma',\phi},\]
where we have chosen $\gamma'\in \Inndiag(E_6(q))\gen{\gamma}$ to be a graph automorphism commuting with $\phi$, and $\PGL_2(q^{1/2}) = C_{\PGL_2(q)}(\phi)$.
Therefore, by item (5) in the proof of Proposition \ref{caseE6},
\[ m_2(K_1) = 2 + m_2( E_6(q) \gen{\gamma',\phi} ) = 2 + 8 = 10.\]
Since $m_2(L\gen{\phi}) \leq m_2(L)+1=10$, we conclude that $m_2(K_1) = m_2(L\gen{\phi})$.

Finally, note that $K_0$ and $K_1$ satisfy $\QD_2$ if $q\neq 3,9$ respectively, by Propositions \ref{casePSL2}, \ref{caseE6} and \ref{prop:productJoinQDp}.
Hence, by Lemmas \ref{lm:pprimequotient} and \ref{lm:inclusion}, $L$ and $L\gen{t}$ satisfy $\QD_2$ if $q\neq 3,9$ respectively.

Therefore every $2$-extension of $E_8(q)$ satisfies $\QD_2$, with the exceptions given in the statement.
This concludes the proof of the proposition.

\end{proof}

\bibliographystyle{abbrv}
\bibliography{references}

\end{document}

%% file: operators.tex
 \def\QQ{\mathbb{Q}}

 \def\A{\mathcal{A}}

 \newcommand{\GF}[1]{\mathbb{F}_{#1}}

 \DeclareMathOperator\GL{GL}
 \DeclareMathOperator\PSL{PSL}
 \DeclareMathOperator\PGL{PGL}
 \DeclareMathOperator\SL{SL}

 \DeclareMathOperator{\GU}{GU}
 
 \DeclareMathOperator{\PGU}{PGU}
 \DeclareMathOperator{\PSU}{PSU}

 \DeclareMathOperator{\PSp}{PSp}
 \DeclareMathOperator{\GO}{GO}
 
 \DeclareMathOperator{\Spin}{Spin}
 \DeclareMathOperator{\HSpin}{HSpin}
 
 \DeclareMathOperator{\POm}{P\Omega}
 
 \DeclareMathOperator\Sz{Sz}

 \DeclareMathOperator{\Sym}{Sym}
 \DeclareMathOperator{\Alt}{Alt}
 \DeclareMathOperator{\Cyclic}{C}
 \DeclareMathOperator{\Dihedral}{D}

 \DeclareMathOperator\Aut{Aut}
 
 \DeclareMathOperator\Out{Out}
 \DeclareMathOperator\Inn{Inn}
 \DeclareMathOperator\Outdiag{Out diag}
 \DeclareMathOperator\Inndiag{Inn diag}
 
 \newcommand{\tq}{\mathrel{{\ensuremath{\: : \: }}}}

 \def\Sp{\mathcal{S}_p} 
 \def\Ap{\mathcal{A}_p}
 
 \DeclareMathOperator\Outposet{{O}}
 \def\HO{\hat{\Outposet}_2}

 \def\QD{(\mathcal{Q}\mathcal{D})}

 \newcommand{\normal}{\trianglelefteq}
 \def\groupiso{\cong}

 \newcommand\gen[1]{\left\langle#1\right\rangle}